\documentclass[12pt]{amsart}
\usepackage{amsmath,amsfonts,amssymb} \usepackage{color}
\usepackage[all,cmtip]{xy}
\usepackage{graphicx}
 \usepackage{youngtab}

\newcommand{\corr}[1]{\langle {#1} \rangle}
\newcommand{\corrr}[1]{\langle\langle {#1} \rangle\rangle}

\newcommand{\bt}{{\bf t}}

 \newcommand{\bZ}{\mathbb{Z}}
 \newcommand{\bC}{\mathbb{C}}
 \newcommand{\pd}{\partial}
\newcommand{\Mbar}{\overline{\mathcal M}}
 \newcommand{\bu}{\mathbf{u}}
\newcommand{\vac}{|0\rangle} \newcommand{\lvac}{\langle 0|}

  \DeclareMathOperator{\res}{res}  
    \DeclareMathOperator{\erf}{erf}

\newcommand{\be}{\begin{equation}}
\newcommand{\ee}{\end{equation}}
\newcommand{\bea}{\begin{eqnarray}}
\newcommand{\eea}{\end{eqnarray}}
\newcommand{\ben}{\begin{eqnarray*}}
\newcommand{\een}{\end{eqnarray*}}

\newcommand{\half}{\frac{1}{2}}

\newtheorem{cor}{Corollary}[section]

\newtheorem{conj}[cor]{Conjecture}
\newtheorem{lem}[cor]{Lemma}
 \newtheorem{prop}[cor]{Proposition}
 \newtheorem{thm}[cor]{Theorem}
\theoremstyle{remark}
 
 \newtheorem{ex}[cor]{Example}

\definecolor{A}{rgb}{.75,1,.75}

\definecolor{green}{rgb}{0,1,0}
\definecolor{yellow}{rgb}{1,1,0}
\definecolor{orange}{rgb}{1,.7,0}
\definecolor{red}{rgb}{1,0,0}
\definecolor{white}{rgb}{1,1,1}

 \makeindex
\begin{document}
\title[Quantum Deformation Theory and Mirror Symmetry]
{Quantum Deformation Theory of the Airy Curve and Mirror Symmetry of a Point}
%\author{ }
%\thanks{ }

\author{Jian Zhou}
\address{Department of Mathematical Sciences\\Tsinghua University\\Beijng, 100084, China}
\email{jzhou@math.tsinghua.edu.cn}

\begin{abstract}
We present a quantum deformation theory of the Airy curve and use it to establish a version of mirror symmetry of a point.
\end{abstract}

\maketitle

\section{Introduction}

``To see a world in a grain of sand" is the first line of a famous poem by William Blake.
For the author of this paper,
the grain of sand is just a point,
and the world he wants to see is mirror symmetry.
More precisely,
he wants to gain some more understanding on mirror symmetry by studying the mirror symmetry of a point.

It may sound absurd to  consider mirror symmetry of a point,
because mirror symmetry usually involves Calabi-Yau 3-folds \cite{CDGP},
or Fano manifolds and their mirror geometries.
As a space, a point is as simple as one can get (except for perhaps the empty space),
and it seems that it does not have any interesting structure that be used to produce
a mirror partner.
Nevertheless,
the Gromov-Witten theory of a point is very rich and does have a mirror theory encoded in the Airy curve:
\be
y = \half x^2,
\ee
that is the subject of this paper.
This simple curve that everyone sees in junior high school turns out to have a rich 
deformation theory that can be used as the mirror of  the theory of 2D topological gravity,
in more than one way.

For early history on mirror symmetry we refer to \cite{Hori-Vafa}.
Let us recall some highlights that are are particularly relevant to  this work.
The original proposal of mirror symmetry relates the deformations of symplectic structure
on one Calabi-Yau 3-fold with the deformations of the complex structure on its mirror
Calabi-Yau 3-fold.
Based on mirror symmetry,
physicists made predictions on counting the number of holomorphic curves
in quintic threefolds in genus zero \cite{CDGP}.
For $g > 0$,
physicists have developed holomorphic anomaly equation to make predictions
in genus one \cite{BCOV1}, genus 2 \cite{BCOV2},
and in genus $g \leq 51$ \cite{Hua-Kle-Qua}.
Mathematically,
these predictions have been proved in genus zero \cite{Givental, LLY} and genus one \cite{Zinger}.
A physical reformulation of the holomorphic anomaly equation was given in \cite{Witten-Q-background}.
A mathematical formulation of quantum BCOV theory has been developed \cite{Costello-Li}.

Parallel to the mirror symmetry of compact Calabi-Yau 3-folds,
one can consider the local mirror symmetry of toric Calabi-Yau 3-folds \cite{CKYZ}.
Based on duality with Chern-Simons theory of link invariants \cite{Witten-CS-Jones, Witten-CS-String},
a physical theory of topological vertex has been developed to compute local Gromov-Witten
invariants in all genera \cite{AMV, AKMV}.
A mathematical theory of the topological vertex has been developed in \cite{LLZ, LLZ2, LLLZ}.

Even though the theory of topological vertex solves the problem of computing local Gromov-Witten invariant
of toric Calabi-Yau 3-folds in principle,
its combinatorial nature does not  match directly  with the deformation theory aspects  of mirror symmetry
of compact Calabi-Yau 3-folds,
therefore its success in all genera does not provide much insight for its compact counterpart,
except for perhaps the Gopakumar-Vafa invariants \cite{GV, Peng, Guo-Zhou1, Guo-Zhou2}.
Eleven years ago, physicists \cite{ADKMV} have proposed a mathematically mysterious approach based on deformation at infinities of local mirror curves.
More recently local mirror curves have been used as the spectral curve to apply
the appearance of Eynard-Orantin topological recursion \cite{EO} in the program of remodelling the local B-models \cite{Marino, BKMP, EO2}.
Deformations of the local mirror curves do not play a role in this topological recursion formalism either.
The renewed interest in local mirror symmetry and local mirror curves motivates the author
to present a mathematical elaboration of the beautiful ideas in \cite{ADKMV},
with an emphasis on the deformation theory aspects.
Through this one can gain more insights about the mirror symmetry in the compact Calabi-Yau case.

In physics Gromov-Witten theory of  a point corresponds to the theory of 2D topological gravity
originally studied using matrix models,
and mathematically they correspond to intersection numbers of $\psi$-classes
on Deligne-Mumford moduli spaces.
Witten's remarkable conjecture relates such numbers
to KdV hierarchy and Virasoro constraints.
Since its first proof by Kontsevich \cite{Kontsevich} it has served as a paradigm for researches
in Gromov-Witten theory and its generalizations.
Recent progresses on topological recursions related to
the Witten-Kontsevich tau-function \cite{Eynard, BCSW,  Zhou1, Zhou2}
have made it clear that the {\em Airy curve}:
$$
y = \half x^2,
$$
plays an important role in this theory.
In this paper we will elaborate on the fact that all information about the correlation
functions of the 2D topological gravity
is encoded in the Airy curve,
by suitable procedures of deformation and quantization,
therefore,
this curve  serves as the mirror geometry in the B-theory for the point.
This is an elaboration of some ideas in \cite{ADKMV}
where some technical details are missing.
This will provide the prototype of quantum deformation theory of a class of curves which we will develop
in subsequent work.
Our ultimate goal will be to develop a version of quantum deformation to unite
the two main driving forces  in Gromov-Witten type theories:
The Witten Conjecture/Kontsevich Theorem
and the computation of Gromov-Witten invariants of quintic Calabi-Yau 3-folds.

Deformation theory was created by Riemann in his famous paper on Abelian functions
where he introduced the notion of moduli spaces of Riemann surfaces.
Deformation theory of compact complex manifolds was developed by Kodaira, Spencer and Kuranishi.
There is also a deformation theory of isolated singularities.
It is amazing that such theories all have found applications in string theory related to mirror symmetry.
In the study of mirror symmetry of quintic Calabi-Yau 3-fold,
deformation of its mirror has been applied to find genus zero free energy.
We will somehow reverse this direction of applications:
We will first use free energy in genus zero to construct the deformation of the Airy curve,
then we will apply the canonical quantization derive the Virasoro constraints satisfied
by free energies in higher genera.
Generalizations to Type A and Type D singularities
are straightforward and will be presented in future publications.
Generalizations to other case are work in progress.

We now summarize the main results of this paper.
We first consider the miniversal deformation
$y = \half x^2 + u_0$ of the Airy curve $y= \half x^2$.
Write $f^2 = 2y$,
then we prove the following identity:
\be
x = f(1-\frac{2u_0}{f^2})^{1/2}
= f - \frac{u_0}{f}
- \sum_{n =0}^\infty \frac{\pd F_0}{\pd u_n}(u_0, 0, \dots) \cdot f^{-(2n+3)},
\ee
where $F_0$ is the genus zero free energy of the theory of 2D topological gravity,
in suitably chosen coordinates $u_n$'s.
Next we construct a special deformation of the Airy curve of the following form:
\be \label{eqn:Special def}
x = f - \sum_{n \geq 0} (2n+1) u_n f^{2n-1} - \sum_{n \geq 0} \frac{\pd F_0}{\pd u_n}(\bu) \cdot f^{-2n-3}.
\ee
We prove that it is uniquely  characterized by the following property:
\be
(x^2(f))_- = 0.
\ee
From \eqref{eqn:Special def} we obtain a special deformation
of the equation of the Airy curve of the form:
$$ 2y = f^2 = c_{0} + c_1 x^2 + c_2 x^4 + c_3 x^6 + \cdots.$$
We regard $W =y$ as the deformed superpotential function and define
\be
\phi_n: = \frac{\pd W}{\pd u_n}.
\ee
They generate an algebra over the ring of  formal series in $u_n$'s,
 closed under multiplications.
We take this as a generalization of the Milnor ring.
Furthermore,
if one defines
\be
\nabla^\lambda_{\frac{\pd}{\pd t_i}} \phi_k := \pd_{t_i} \phi_k + \lambda \phi_i \cdot \phi_k.
\ee
then the operators $\nabla^\lambda_{\frac{\pd}{\pd t_i}}$ define a family of flat connections,
i.e.,
\be
\nabla^\lambda_{\frac{\pd}{\pd t_i}} \nabla^\lambda_{\frac{\pd}{\pd t_j}}\phi_k
= \nabla^\lambda_{\frac{\pd}{\pd t_j}} \nabla^\lambda_{\frac{\pd}{\pd t_i}}\phi_k,
\ee
for $i,j,k \geq 0$ and all $\lambda \in \bC$.
We then define $\corrr{\phi_{j_1}, \cdots, \phi_{j_n}}_0$ as in Landau-Ginzburg theory
and show that
\be
\corrr{\phi_{j_1}, \cdots, \phi_{j_n}}_0
= \frac{\pd^n F_0}{\pd t_{j_1} \cdots \pd t_{j_n}}.
\ee
To extend the picture to arbitrary genera,
we quantize the above picture.
We endow the space   of series of the form
\be
\sum_{n =0}^\infty (2n+1) \tilde{u}_n z^{(2n-1)/2}
+ \sum_{n =0}^\infty \tilde{v}_n z^{-(2n+3)/2}
\ee
 the following symplectic structure:
\be
\omega = \sum_{n =0}^\infty   d\tilde{u}_n \wedge d \tilde{v}_n.
\ee
Take the natural polarization that $\{q_n = \tilde{u}_n\}$ and $\{p_n = \tilde{v}_n\}$,
one can consider the canonical quantization:
\be
\hat{\tilde{u}}_n = \tilde{u}_n \cdot, \;\;\; \hat{\tilde{v}}_n = \frac{\pd}{\pd \tilde{u}_n}.
\ee
Corresponding to the field $x$,
we consider the following fields of operators on the Airy curve:
\be
\hat{x}(z) = - \sum_{m \in \bZ} \beta_{-(2m+1)} z^{m-1/2}
= - \sum_{m \in \bZ} \beta_{2m+1} z^{-m-3/2}
\ee
where $f = z^{1/2}$ and the operators $\beta_{2k+1}$ are defined by:
\be
\beta_{-(2k+1)} = (2k+1) \tilde{u}_k \cdot, \;\;\;\; \beta_{2k+1} = \frac{\pd}{\pd \tilde{u}_k}.
\ee
We define a notion of regularized products $\hat{x}(z)^{\odot n}$
and show that they are related to the normally ordered products $:\hat{x}(z)^n:$ by Bessel numbers.
Then we show that the DVV Virasoro constraints satisfied by the
Witten-Kontsevich tau-function
is just the following equation:
\be
\hat{x}(z)^{\odot 2} Z_{WK} = 0.
\ee
We conjecture that
\be
\hat{x}(z)^{\odot 2n} Z_{WK} = 0
\ee
for $n > 1$.
In forthcoming work,
we will generalize such results to deformations of isolated singularities of  Type A and Type D.

Roughly speaking,
what we mean by quantum deformation theory
is a deformation theory whose moduli space  encodes the information of genus zero free energy 
on the big phase space.
It is then necessary that the moduli space is infinite-dimensional.
Furthermore,
we require that the moduli space  admits a natural quantization from which one can produce constraints that determines the free energy in 
all genera.
This is the point of view that we will take to understand mirror symmetry in subsequent work, 
both for noncompact and compact Calabi-Yau manifolds.

The rest of the paper is arranged as follows.
In  \S \ref{sec:Prelim} we fix some terminologies and recall some combinatorial preliminaries.
We recall the Witten Conjeture/Kontsevich Theorem and DVV Virasoro constraints in \S \ref{sec:A-Theory},
some explicit numerical computations based on formal series solution of the inviscid Burgers' equation
are also presented.
We construct in \S 4 the special deformation of the Airy curve based on
genus zero Gromov-Witten invariants of a point.
In \S 5 we study the Landau-Ginzburg type theory associated to the miniversal deformation of the Airy curve
and make extensions to the special deformation in \S 6.
In \S 7 and \S 8 we present the quantum deformation of the Airy curve by quantization $\hat{x}$ of the field $x$,
define and study its regularized products relating to Virasoro constraints and its generalizations.

\section{Preliminaries}
\label{sec:Prelim}

In this paper we will discuss the mirror symmetry between two ``field theories".
We will use this section to make a provisional definition of what we mean by a field theory
to make our framework as simple as possible.
We will also collect some combinatorial techniques crucial for our computations and proofs.

\subsection{Some terminologies of field theory}

The physical background for our definition is the topological matters coupled with 2D topological gravity.
The definition that we will take below does not reflect all the interesting aspects of such theories,
it only serves to suit for the discussions in this work.
The reader may consult \cite{Kontsevich-Manin} for a more general definition.

First,
we need a finite-dimensional complex vector space $V$.
It will be called the {\em small phase space},
and we will call a vector $v \in V$ a {\em primary field}.
It will also be referred to as a {\em matter field}.
The small phase space is supposed to be graded by rational numbers:
\be
V = \bigoplus_{i=1}^k V_{w_i},
\ee
where $w_1 < w_2 < \cdots < w_k$ are distinct rational numbers.
If $v \in V_{w_j}$,
then we say {\em $v$ is a homogeneous filed of degree $w_j$},
and write
$$\deg v = w_j.$$
By the {\em big phase space}
we mean $\bC[z]\otimes_\bC V$.
We will write this space as $V[z]$.
(Usually one reserves this notation for $V$ be a commutative ring so that $V[z]$ is also a ring.
A field theory often naturally makes $V$ a ring.)
An element of this space can be written as
\ben
&& v_0 + v_1z + \cdots + v_nz^n,
\een
where $v_0, \dots, v_n\in V$.
For $v \in V$,
the element $vz^n$ will be called the {\em $n$-th gravitational descendant (field)} of $v$.
Following Witten,
we will also denote $vz^n$ by $\tau_n(v)$.
The grading on $V$ induces a grading $V[z]$:
$$V[z] = \bigoplus_{n \geq 0} \bigoplus_{i=1}^k V_{w_i}z^n,
$$
i.e.,
if $v\in V_{w_j}$ has $\deg v = w_j$,
then one sets
$$\deg \tau_n(v) = w_j + n.$$
By a field theory of the topological matters coupled to 2D topological gravity,
we simply mean a collection of {\em correlators}: For $2g-2+n > 0$,
\be
\corr{\cdot, \cdots, \cdot}_g: V[z]^{\otimes n} \to \bC,
\ee
$$(\tau_{m_1}(v_1), \dots, \tau_{m_n}(v_n)) \mapsto \corr{\tau_{m_1}(v_1), \dots, \tau_{m_n}(v_n)}_g.$$
They are required to satisfy the following two conditions:
The correlators are symmetric in its arguments, i.e.,
\be
\corr{\tau_{m_{\sigma(1)}}(v_{\sigma(1)}), \dots, \tau_{m_{\sigma(n)}}(v_{\sigma(n)})}_g
= \corr{\tau_{m_1}(v_1), \dots, \tau_{m_n}(v_n)}_g,
\ee
for any permutation $\sigma: \{1, \dots, n\} \to \{1, \dots, n\}$;
the correlators satisfy the following {\em selection rule}:
Suppose that $v_1, \dots, v_n$ are homogeneous fields,
then
$$
\corr{\tau_{m_1}(v_1), \dots, \tau_{m_n}(v_n)}_g = 0
$$
except for
\be
\sum_{i=1}^n (m_i +w_i) = 3g-3+n.
\ee
Take a basis $\{e_1, \dots, e_m\}$ of $V$ consisting of homogeneous elements,
then one gets a basis $\{\tau_n(e_i)\}_{n \geq 0, 1\leq i \leq m}$ of $V[z]$.
The corresponding linear coordinates $\{t_{n,i}\}_{n \geq 0, 1\leq i \leq m}$
will be called {\em coupling constants}.
Define the {\em free energies in genus $g$} by:
\be
F_0(\bt) = \sum_{n \geq 3} \sum_{\substack{m_1, \dots, m_n \geq 0 \\ 1\leq i_1, \dots, i_n \leq m}}
\frac{t_{m_1, i_1} \cdots t_{m_n, i_n}}{n!}
\corr{\tau_{m_1}(e_{i_1}), \dots, \tau_{m_n}(e_{i_n})}_0,
\ee
and for $g \geq 1$,
\be
F_g(\bt) = \sum_{n \geq 1} \sum_{\substack{m_1, \dots, m_n \geq 0 \\ 1\leq i_1, \dots, i_n \leq m}}
\frac{t_{m_1, i_1} \cdots t_{m_n, i_n}}{n!}
\corr{\tau_{m_1}(e_{i_1}), \dots, \tau_{m_n}(e_{i_n}))}_g.
\ee
The {\em partition function} is defined by:
\be
Z(\bt; \lambda) = \exp \sum_{ g\geq 0} \lambda^{2g-2} F_g(\bt).
\ee
Define the {\em deformed $n$-point correlators} by:
\be
\corrr{\tau_{m_1}(e_{i_1}), \dots, \tau_{m_n}(e_{i_n})}_g(\bt)
=  \frac{\pd^n}{\pd t_{m_1, i_1} \cdots \pd t_{m_n, i_n}} F_g(\bt).
\ee
From this definition it is clear that
\be
\begin{split}
& \corrr{\tau_{m_1}(e_{i_1}), \dots, \tau_{m_n}(e_{i_n}), \tau_{m_{n+1}}(e_{i_{n+1})}}_g(\bt) \\
= & \frac{\pd}{\pd t_{m_{n+1}, i_{n+1} } }
\corrr{\tau_{m_1}(e_{i_1}), \dots, \tau_{m_n}(e_{i_n})}_g(\bt).
\end{split}
\ee
Hence to determine the theory,
it suffices to determine the deformed one-point functions
$\corrr{\tau_m(e_i)}_g(\bt)$.

\subsection{Lagrangian inversion}

Suppose that we have two formal power series:
\bea
&& y = x + a_2 x^2 + a_3 x^3 + \cdots, \\
&& x = y + b_2 y^2 + b_3 y^3 + \cdots
\eea
are compositional inverse to each other.
Then their coefficients are related to each other by the Lagrange inversion formula:
\bea
&& a_n = \frac{1}{n} \res (\frac{dy}{x^n}) = \frac{1}{n} (y+b_2y^2 + \cdots)^{-n}|_{y^{-1}},   \\
&& b_n = \frac{1}{n} \res (\frac{dx}{y^n}) = \frac{1}{n} (x+a_2x^2 + \cdots)^{-n}|_{x^{-1}},
\eea
The following are some explicit examples:
\ben
&& a_2 = - b_2, \\
&& a_3 = - b_3 + 2 b_2^2, \\
&& a_4 = - b_4 + 5b_2b_3 - 5 b_2^3.
\een
If one assigns
\be
\deg b_n = 1- n, \;\; n \geq 2,
\ee
then $a_n$ is a weighted homogeneous polynomial in $b_2, \dots, b_n$ of degree $1-n$.
We will refer to $\{a_n\}_{n \geq 2}$ and $\{b_n\}_{n \geq 2}$ as Lagrange dual of each other.

\section{The A-Theory: Intersection Theory on Moduli Spaces of Algebraic Curves}

\label{sec:A-Theory}

In this section,
we recall the A-theory of a point,
i.e.,
the intersection theory of $\psi$-classes on Deligne-Mumford moduli spaces $\Mbar_{g,n}$.
In the physics literature (see e.g.  \cite{Witten1}),
this corresponds to the 2D topological gravity.
Main features of this theory are the Witten Conjecture/Kontsevich Theorem \cite{Witten1, Kontsevich}
and the Virasoro constraints \cite{DVV-Virasoro}.
These results tie this theory to the KdV hierarchy,
and  hence the genus zero part to the dispersionless limit of the KdV hierarchy.
One can reduce the genus zero part of the theory to inviscid Burger's equation
and its series solution by Lagrange inversion.
We also derive formula \eqref{eqn:Der-F0} which is the starting point
for mirror symmetry with quantum deformation theory of the Airy curve.

\subsection{Free energy and partition function of 2D topological gravity}

The moduli spaces for 2D topological gravity are the Deligne-Mumford spaces $\Mbar_{g,n}$
of stable algebraic curves with $n$ marked points.
Let $\psi_1, \dots, \psi_n$ be the first Chern classes of the cotangent line bundles
corresponding to the $n$ marked points.
The correlators of the 2D topological gravity are defined as the following intersection numbers:
\be
\corr{\tau_{a_1}, \cdots, \tau_{a_n}}_g :=
\int_{\Mbar_{g,n}} \psi_1^{a_1} \cdots \psi_n^{a_n}.
\ee
The correlator $\corr{\tau_{a_1}, \cdots, \tau_{a_n}}_g \neq 0$ only if
\be \label{eqn:SelRule}
a_1 + \cdots + a_n = 3g-3 + n.
\ee
This is due to the fact that
\be
\dim \Mbar_{g,n} = 3g-3+n.
\ee

The free energy of 2D topological gravity is defined by
\be
F(\bt; \lambda) = \sum_{g \geq 0} \lambda^{2g-2} F_g(\bt).
\ee
where the genus $g$ part of the free energy is defined by:
\be
F_g(\bt) = \corr{\exp \sum_{a \geq 0} t_a \tau_a }_g
= \sum_{n \geq 0} \frac{1}{n!} \sum_{a_1, \dots, a_n \geq 0} t_{a_1} \cdots t_{a_n}
\corr{ \tau_{a_1}, \cdots, \tau_{a_n}}_g.
\ee
We will assign the following grading:
\be
\deg t_i = 2 - 2i, \;\;\; i =0,1,\dots.
\ee
By \eqref{eqn:SelRule},
$F_g(\bt)$ is weighted homogeneous of degree
\be
\deg F_g = 6-6g.
\ee
The partition function of 2D topological gravity,
often referred to as the  Witten-Kontsevich tau-function, is defined by
\be
Z_{WK}(\bt, \lambda) = \exp F(\bt, \lambda).
\ee
We assign
\be
\deg \lambda = 3
\ee
so that $Z_{WK}$ is weighted homogenous of degree $0$.

\subsection{Witten Conjecture/Kontsevich Theorem}

Witten \cite{Witten1} conjectured that $Z_{WK}$ is a tau-function of the KdV hierarchy.
More precisely,
define a sequence $\{ R_n\}$ of differential polynomials in $u$ as follows:
\be
\begin{split}
& R_1 = u, \\
& \frac{\pd}{\pd x} R_{n+1} = \frac{1}{2n+1} \biggl(\pd_xu \cdot R_n
+ 2u \cdot \pd_x R_n + \frac{\lambda^2}{ 4} \pd_x^3R_n \biggr).
\end{split}
\ee
For example,
\ben
\pd_x R_2 & = & u \cdot \pd_x u + \frac{1}{12} \pd_x^3 u, \\
R_2 & = & \frac{1}{2} u^2 + \frac{1}{12} \pd_x^2u,  \\
\pd_x R_3 & = & \frac{1}{2} u^2\cdot \pd_x u + \frac{1}{12} u \cdot \pd_x^3u + \frac{1}{6}\pd_x u \cdot \pd_x^2u + \frac{1}{240}\pd_x^5u, \\
R_3 & = & \frac{1}{6} u^3 + \frac{1}{12} u \cdot \pd_x^2 u + \frac{1}{24} (\pd_x u)^2 + \frac{1}{240} \pd_x^4 u.
\een
Then
\be
u = \lambda^2 \frac{\pd F}{\pd^2 t_0^2}
\ee
satisfies the following sequence of equations:
$$\pd_{t_n} u = \pd_{t_0} R_{n+1},$$
where $t_0 = x$.

He also pointed out that together with the string equation:
\be
\corr{\tau_0 \tau_{a_1} \cdots \tau_{a_n}}_g  =
\sum_{i=1}^n \corr{\tau_{a_1} \cdots \tau_{a_i-1} \cdots \tau_{a_n}}_g,
\ee
this hierarchy of nonlinear differential equations uniquely determines $Z_{WK}$ from the initial values:
\ben
\int_{\Mbar_{0,3}} 1 = 1, \qquad \int_{\Mbar_{1,1}} \psi_1 = \frac{1}{24},
\een
After Kontsevich \cite{Kontsevich},
there have appeared many proofs of this conjecture.

\subsection{Dispersionless limits and reduction to inviscid Burgers' equation}
Write
$$u = u_{(0)} + \lambda^2 u_{(1)} + \cdots,$$
Then from the KdV hierarchy one gets the following sequence of equations for $u_{(0)}$:
\be \label{eqn:Dispersionless-Hierarchy}
\pd_{t_n} u_{(0)} = \pd_{t_0} \frac{u_{(0)}^{n+1}}{(n+1)!}.
\ee
Since these flows commute with each other,
one can solve them separately and inductively.

The case of $n =1$ is the inviscid Burgers' equation:
\be
\pd_t u = \pd_x \frac{u^2}{2}.
\ee
The next result shows that all higher equations in the hierarchy \eqref{eqn:Dispersionless-Hierarchy}
can be reduced to it.

\begin{prop}
Suppose that $u(x,t)$ satisfies the equation
\be
\pd_t u = \pd_t \frac{u^{n+1}}{(n+1)!}
\ee
for some $n>1$,
then $v = \frac{u^n}{n!}$ satisfies:
\be
\pd_t v = \pd_x \frac{v^2}{2}.
\ee
\end{prop}

\begin{proof}
Easy computations.
\end{proof}

\subsection{Series solutions of inviscid Burgers' equation by Lagrangian inversion}

Recall that the inviscid Burgers' equation can be solved by the method of characteristics.
Suppose that the initial value $u(x,0) = u_0(x)$ is given.
If $x(t)$ with $x(0) = x_0$ is a solution of the ordinary differential equation
\be \label{eqn:Characteristic}
\frac{d}{dt} x(t) = - u(x(t), t),
\ee
where $u=u(x,t)$ satisfy $u_t = uu_x$,
then
\ben
\frac{d}{dt} u(x(t), t) = u_x \cdot \frac{d}{dt} x(t) + u_t
= - u_x u + u_t = 0,
\een
and so
\be \label{eqn:Characteristic2}
u(x(t), t) = u(x(0), 0) = u(x_0, 0) = u_0(x_0).
\ee
Now  from \eqref{eqn:Characteristic},
\ben
\frac{d}{dt} x(t) = - u(x(t), t) = - u_0(x_0),
\een
and so
\be
x(t) = x_0 - t u_0(x_0).
\ee
Plug this back into \eqref{eqn:Characteristic2}
and change $x_0$ into $x$:
\be
u(x - t u_0(x), t) = u_0(x).
\ee
To see this provides formal series solutions,
we set
\be
A(x, t) = x + t \cdot u(x, t).
\ee
Then one has
\ben
&& A(x- t u_0(x), t) = (x- t u_0(x)) + t u(x_0-t u_0(x), t) \\
& = & x - t u_0(x) + t u_0(x) = x.
\een
So we get the following:

\begin{prop} \label{prop:Lagrange}
Suppose that $A(y,t)$ is a formal series in $y$ which is the compositional inverse of
the equation
\be \label{eqn:Lagrange}
y = x - t u_0(x),
\ee
i.e., $A(x-tu_0(x), t) = x$,
then
\be
u(x,t) = \frac{A(x,t)- x}{t}
\ee
is a formal series solution of the inviscid Burger's equation
with initial condition $u_0(x)$:
\be
\pd_t u = \pd_x \frac{u^2}{2}, \;\; u(x, 0) = u_0(x).
\ee
\end{prop}

One can use Lagrange inversion to solve \eqref{eqn:Lagrange}.
Let
$$A(x,t) = x + \sum_{k \geq 2} a_k(t) x^k,$$
then one has:
\be
a_k(t) = \frac{1}{k} (x-tu(x,t))^{-k}|_{x^{-1}}.
\ee
Sometimes one can directly solve  \eqref{eqn:Lagrange} by explicit expressions.
The following are some examples.

\begin{ex}
When $u_0(x) = x$,
the equation
\ben
&& y = x - tx
\een
can be solved by
$$x = \frac{y}{1-t},$$
i.e.,
$$A(x,t) = \frac{x}{1-t}.$$
It follows that
\ben
u(x,t) = \frac{A(x,t) -x}{t} = \frac{x}{1-t}.
\een
One can easily checked that it is a solution.
\end{ex}

\begin{ex}
When $u_0(x) = x^2 $,
the equation
\ben
&& y = x - tx^2
\een
can be solved by
$$x = \frac{1 - (1-4ty)^{1/2}}{2t},$$
i.e.,
$$A(x,t) = \frac{1 - (1-4tx)^{1/2}}{2t}.$$
It follows that
\ben
u(x,t) = \frac{A(x,t) -x}{t} = \frac{1 -2tx- (1-4tx)^{1/2}}{2t^2}.
\een
One can easily checked that it is a solution.
The coefficients of the series expansion of $u(x,t)$ are Catalan numbers:
\be
u(x,t) = \sum_{n=0}^\infty \frac{(2n+2)!}{(n+1)!(n+2)!} t^n x^{n+2}.
\ee
\end{ex}

In general we have:

\begin{lem} \label{lm:Burgers-Solution}
For $m =1, 2, \dots$,
the following equation
\be
\pd_t u = \pd_x \frac{u^2}{2}, \;\; u(x, 0) = c x^m
\ee
is solved by the following generating series of generalized Catalan numbers:
\be
u(x, t) =  x^m \sum_{k=1}^\infty \frac{c^k}{(m-1)k+1} \binom{mk}{k} (x^{m-1}t)^{k-1}.
\ee
\end{lem}

\begin{proof}
One can use Lagrange inversion to solve
\ben
y = x -c tx^m,
\een
i.e.,
let
$x = A(y, t) = y + \sum_{n \geq 2} a_n y^n$,

\ben
a_n & = & \frac{1}{2\pi \sqrt{-1}} \oint \frac{x}{y^{n+1}} dy
= - \frac{1}{n} \cdot \frac{1}{2\sqrt{-1}} \oint x d \frac{1}{y^n} \\
& = & \frac{1}{n} \cdot \frac{1}{2\sqrt{-1}} \oint \frac{1}{y^n} d x
= \frac{1}{n} \cdot \frac{1}{2\sqrt{-1}} \oint \frac{1}{(x- c t x^m)^n} d x \\
& = &  \frac{1}{n} \cdot \frac{1}{2\sqrt{-1}} \oint \frac{1}{x^n}
\sum_{k=0}^\infty \frac{1}{(1- c t x^{m-1})^n} d x \\
& = & \frac{1}{n} \cdot \frac{1}{2\sqrt{-1}} \oint \frac{1}{x^n}
\sum_{k=0}^\infty \binom{-n}{k} (-c t x^{m-1})^{k} d x.
\een
It is now clear that $a_n$ only when $n=k(m-1)+1$ for some $k \geq 1$,
\ben
a_{k(m-1)+1} & = & \frac{1}{k(m-1)+1} \binom{-k(m-1)-1}{k}(-c t)^{k} \\
& = & \frac{1}{k(m-1)+1} \binom{km}{m} c^k t^k.
\een
hence
\ben
A(x, t) = x + \sum_{k \geq 1} \frac{1}{k(m-1)+1} \binom{km}{m} c^k t^k x^{k(m-1)+1}.
\een
The result then follows.
\end{proof}

\subsection{Some explicit calculations}

Now we use the above method to carry out some explicit calculations.
We will use the following notation when there is no danger of  confusions:
We keep only the relevant variables, e.g.,
$u(t_0, t_2,  t_5)$ means $u(\bt)$ restricted to $t_i =0, i \neq 0, 2, 5$.

First we solve
\ben
\frac{\pd u_{(0)}}{\pd t_n} = \pd_{t_0} \frac{u_{(0)}^{n+1}}{(n+1)!}, \;\;\; u(t_0, 0, \dots) = t_0.
\een
This can be transformed to
\ben
v= \frac{u_{(0)}^n}{n!} , \;\;\;
\frac{\pd v}{\pd t_n} = \pd_{t_0} \frac{v^2}{2}, \;\;\;
v(t_0, 0, \dots) = \frac{t_0^n}{n!}.
\een
By Lemma \ref{lm:Burgers-Solution},
\ben
v(t_0, t_n)
= \frac{t_0^n}{n!} \sum_{k=1}^\infty \frac{1}{(n-1)k+1}
\binom{n k}{k} \biggl(\frac{t_0^{n-1} t_n }{n!} \biggr)^{k-1}.
\een
Then we have
\be
u_{(0)}(t_0, t_n) = t_0 \biggl[ \sum_{k=1}^\infty \frac{1}{(n-1)k+1}
\binom{n k}{k} \biggl(\frac{t_0^{n-1} t_n}{n!} \biggr)^{k-1} \biggr]^{1/n}.
\ee
It turns out that we have
\be \label{eqn:u(t0,tn)}
u_{(0)}(t_0, t_n) = t_0 \biggl[1 + \sum_{k=1}^\infty \frac{1}{(n-1)k+1}
\binom{n k}{k} \biggl(\frac{t_0^{n-1} t_n}{n!} \biggr)^{k} \biggr]
\ee
This can be proved by using the following identity:
\be
\biggl[ \sum_{k=1}^\infty \frac{z^{k-1}}{(n-1)k+1} \binom{n k}{k}  \biggr]^{1/n}
=  1+ \sum_{k=1}^\infty \frac{z^k}{(n-1)k+1} \binom{n k}{k}.
\ee
One can prove this identity by Lagrange inversion as follows.
Let $y = z  + \sum_{k \geq 2} a_n z^k$ satisfy
\be
(y/z)^{1/n} = 1 + y,
\ee
i.e.
$$z = \frac{y}{(1+y)^n}.$$
Then
\ben
a_k & = & \frac{1}{2\pi \sqrt{-1}} \oint \frac{y}{z^{k+1}} dy
= \frac{1}{k} \frac{1}{2\pi \sqrt{-1}} \oint \frac{1}{z^k} dy \\
& = & \frac{1}{k} \frac{1}{2\pi \sqrt{-1}} \oint \frac{(1+y)^{kn}}{y^k} dy
= \frac{1}{k} \binom{kn}{k-1} \\
& = & \frac{1}{(k-1)n+1} \binom{kn}{k}.
\een
As straightforward consequence of string equation is that
\be
\pd_{t_n} F_0(t_0, 0, \dots) = \frac{t_0^{n+2}}{(n+2)!}.
\ee
This matches with \eqref{eqn:u(t0,tn)}.

\begin{prop}
The sequence $u_{(0)}(t_0)$, $u_{(0)}(t_0,t_1)$, $\dots$ satisfies the following relations:
\be
u_{(0)}(x, t_1, \dots, t_n)
= u_{(0)}(x + \frac{t_n}{n!} u_{(0)}^n(x,t_1, \dots, t_n), t_1, \dots, t_{n-1}).
\ee
\end{prop}

\begin{proof}
Set
\be
v(x,t_1, \dots, t_n) = \frac{1}{n!} u_{(0)}^n(t_0=x, t_1, \dots, t_n).
\ee
Then we have
\ben
&& \pd_{t_n} v(x, t_1, \dots, t_n) = \pd_x(\frac{v^2(x,t_1, \dots, t_n)}{2}), \\
&& v(x, t_1, \dots, t_{n-1}) = \frac{1}{n!} u_{(0)}^n(x, t_1, \dots, t_{n-1}).
\een
Then by Proposition \label{prop:Lagrange},
\be
v(x,t_1, \dots, t_n) = \frac{A(x, t_1, \dots, t_n) - x}{t_n},
\ee
where
$A(x,t_1, \dots, t_n)$ satisfies:
\be
A(x, t_1, \dots, t_n) - t_n v(A(x, t_1, \dots, t_n), t_1, \dots, t_{n-1}) = x.
\ee
The proof can be finished by noting:
\ben
&& A(x,t_1, \dots, t_n) = x + t_n \cdot v(x, t_1, \dots, t_n) \\
&& \;\;\; = x + \frac{t_n}{n!} u^n_{(0)}(x, t_1, \dots, t_n), \\
&& v(A(x, t_1, \dots, t_n), t_1, \dots, t_{n-1}) \\
& = & \frac{1}{n!} u^n_{(0)}(A(x, t_1, \dots, t_n), t_1, \dots, t_{n-1}) \\
& = & \frac{1}{n!} u^n_{(0)}(x + \frac{t_n}{n!} u^n_{(0)}(x, t_1, \dots, t_n), t_1, \dots, t_{n-1}).
\een
\end{proof}

Using the above method,
one finds $u_{(0)}(t_0,\dots, t_n)$ and $F_0(t_0, \dots, t_n)$ recursively.
For example, first,
\ben
&& u_{(0)}(t_0) = t_0, \;\;\;  F_0(t_0) = \frac{t_0^3}{6};
\een
next,
\ben
u_{(0)}(t_0,t_1) = u_{(0)}(t_0 + t_1 u_{(0)}(t_0, t_1)) = t_0 + t_1 u_{(0)}(t_0, t_1),
\een
hence
\ben
&& u_{(0)}(t_0, t_1) = \frac{t_0}{1-t_1}, \;\;\;\; F_0(t_0,t_1) = \frac{t_0^3}{6(1-t_1)}.
%%%%%%%% && \pd_{t_0} F_0(t_0,t_1) = \frac{t_0^2}{2(1-t_1)}, \;\;\; \pd_{t_1} F_0(t_0,t_1) = \frac{t_0^3}{6(1-t_1)^2}.
\een
Going to the next level,
\ben
u_{(0)}(t_0,t_1, t_2) & = & u_{(0)}(t_0 + \half t_2 u^2_{(0)}(t_0, t_1, t_2), t_1) \\
& = & \frac{t_0 + \half t_2 u^2_{(0)}(t_0, t_1, t_2)}{1-t_1}.
\een
It can be solved explicitly as follows:
\ben
&& u_{(0)}(t_0, t_1, t_2)  =  \frac{1-t_1- ( (1-t_1)^2 - 2t_0t_2 )^{1/2} }{t_2} ,
\een
and so
\be \label{eqn:F0(t0,t1,t2)}
F_0(t_0,t_1,t_2) = \frac{(1-t_1)^5}{15t_2^3} \biggl(1
- \frac{5t_0t_2}{(1-t_1)^2} +  \frac{15t_0^2t_2^2}{2(1-t_1)^4}
- \biggl(1 - \frac{2t_0t_2}{(1-t_1)^2} \biggr)^{5/2} \biggr).
\ee
The series expansion of $u_{(0)}(t_0, t_1, t_2)$ is given by:
\ben
&& u_{(0)}(t_0, t_1, t_2)  =  \frac{1-t_1}{t_2} \biggl[ 1- \biggl( 1- \frac{2t_0t_2}{(1-t_1)^2} \biggr)^{1/2} \biggr] \\
& = & \frac{t_0}{1-t_1} +\frac{1}{2} \frac{t_0^2t_2}{(1-t_1)^3}
+ \frac{1}{2} \frac{t_0^3t_2^2}{(1-t_1)^5} + \frac{5}{8} \frac{t_0^4t_2^3}{(1-t_1)^7} + \cdots,
\een

In general,
one can use Lagrange inversion recursively as follows:
\be
u_{(0)}(t_0 =x, t_1, \dots, t_n)
= \biggl(n! \frac{A(x,t_1, \dots, t_n) - x}{t_n} \biggr)^{1/n},
\ee
where
$$A(x, t_1, \dots, t_n) = x +  \sum_{k \geq 2} a_k x^k,$$
with the coefficients $a_k$ given by
\be
a_k  = \frac{1}{k} (x - \frac{t_n}{n!} u^n_{(0)}(x, t_1, \dots, t_{n-1}))^{-k}|_{x^{-1}}.
\ee
One can easily use a compute algebra system to automate the calculations using these formulas.

\subsection{Virasoro constraints}

Dijkgraaf, E. Verlinde, H .Verlinde \cite{DVV-Virasoro} and independently Fukuma, Kawai, Nakayama \cite{Fukuma-Kawai-Nakayama1}
showed that the Witten-Kontsevich tau-function can also be uniquely determined by
a sequence of linear differential equations called the Virasoro constraints:
\be \label{eqn:Virasoro}
\frac{\pd}{\pd u_{n+1}} Z_{WK} = \hat{L}_n Z_{WK}, \quad n \geq -1,
\ee
where the operators $\hat{L}_n$ are defined by:
\be
\hat{L}_n = \sum_{k=0}^\infty (2k+1) u_k \frac{\pd}{\pd u_{k+n}}+ \frac{\lambda^2}{2}
\sum_{k=0}^{n-1} \frac{\pd^2}{\pd u_k\pd u_{n-k-1}} + \frac{u_0^2}{2\lambda^2} \delta_{n,-1}
+ \frac{\delta_{n,0}}{8}.
\ee
Here we have made the following change of coordinates:
\be
t_k = (2k+1)!! u_k.
\ee
In terms of the free energy,
\bea
&& \frac{\pd F}{\pd  u_0}
= \sum_{k=1}^\infty (2k+1) u_k\frac{\pd F}{\pd u_{k-1}} + \frac{u_0^2}{2\lambda^2},  \label{eqn:Virasoro-1} \\
&& \frac{\pd F}{\pd u_1}
= \sum_{k=0}^\infty (2k+1) u_k\frac{\pd F}{\pd u_{k}} + \frac{1}{8},  \label{eqn:Virasoro0}\\
&& \frac{\pd F}{\pd u_n}
= \sum_{k=0}^\infty (2k+1) u_k \frac{\pd F}{\pd u_{k+n-1}} \label{eqn:VirasoroN} \\
&& \;\;\;\;\;\; + \frac{\lambda^2}{2} \sum_{k=0}^{n-2} \biggl(\frac{\pd^2 F}{\pd u_k\pd u_{n-k-2}}
+ \frac{\pd F}{\pd u_k} \frac{\pd F}{\pd u_{n-k-2}} \biggr), \;\;\; n \geq 2. \nonumber
\eea
In particular,
by comparing the coefficients of $\lambda^{-2}$ on both sides of these equations:
\bea
&& \frac{\pd F_0}{\pd  u_0}
= \sum_{k=1}^\infty (2k+1) u_k\frac{\pd F_0}{\pd u_{k-1}} + \frac{u_0^2}{2}, \\
&& \frac{\pd F_0}{\pd u_1}
= \sum_{k=0}^\infty (2k+1) u_k\frac{\pd F_0}{\pd u_{k}}, \\
&& \frac{\pd F_0}{\pd u_n}
= \sum_{k=0}^\infty (2k+1) u_k \frac{\pd F_0}{\pd u_{k+n-1}}
 + \frac{1}{2} \sum_{k=0}^{n-2}
\frac{\pd F_0}{\pd u_k} \frac{\pd F_0}{\pd u_{n-k-2}} , \;\;\; n \geq 2.
\eea
Now we take $u_i = 0$ for $i \geq 1$ in these equations,
and set $f_n = \frac{\pd F_0}{\pd u_n}(u_0, 0, \dots)$,
we get:
\bea
&& f_0 =  \frac{u_0^2}{2}, \\
&& f_1 = u_0 f_0, \\
&& f_n = u_0 f_{n-1}
 + \frac{1}{2} \sum_{k=0}^{n-2} f_k f_{n-k-2} , \;\;\; n \geq 2.
\eea
After taking the generating series $f(t) = f_0 + f_1 t + \cdots$,
one can get from the above recursion relations the following equation:
\be
f(t) = \frac{u_0^2}{2} + u_0t \cdot f(t) + \half t^2\cdot f(t)^2.
\ee
One can easily find the following explicit expression for $f(t)$:
\be
f(t) = \frac{(1 - u_0 t)- (1-2 u_0t)^{1/2}}{t^2}
= \sum_{ n\geq 0} \frac{(2n+1)!!}{(n+2)!2^{2n+3}} u_0^{n+2} t^n,
\ee
and so
\be \label{eqn:Pd-un-F0}
\frac{\pd F_0}{\pd u_n}(u_0, 0, \dots) = \frac{(2n+1)!!}{(n+2)!} u_0^{n+2}.
\ee
We note the above result can be reformulated as follows:
\be
z (1- \frac{2u_0}{z^2})^{1/2} = z -  \frac{u_0}{z} - \sum_{n =0}^\infty \frac{f_n}{z^{2n+3}},
\ee
or
\be \label{eqn:Der-F0}
z (1- \frac{2u_0}{z^2})^{1/2} = z - \frac{u_0}{z}
- \sum_{n =0}^\infty \frac{\pd F_0}{\pd u_n}(u_0, 0, \dots) \cdot z^{-(2n+3)}.
\ee
This will play an important role in the next section.

\section{The Mirror Geometry: Special Deformation of the Airy Curve}

\label{sec:Deformation}

From the point of view of Eynard-Orantin topological recursions,
it has been clear that
the Airy curve can serve as the mirror curve for the Gromov-Witten theory of a point \cite{Eynard, BCSW, Zhou1, Zhou2}.
In this section we will study some special deformation of the Airy  curve
constructed from Gromov-Witten invariants of a point.

\subsection{Miniversal deformation of the Airy curve}

The Airy curve is the plane algebraic curve given by the equation:
\be
y =\half x^2
\ee
Consider the restriction of the projection $\pi: \bC^2 \to \bC$, $(x, y) \mapsto y$ to this curve,
$(x,y) = (0, 0)$ is an isolated singularity of type $A_1$.
Its miniversal deformation is given by
\be
y = \half x^2 + u_0.
\ee
To unravel the information hidden in this simple formula,
let $f$ be the Laurent series such that $f^2= 2 y = x^2 + 2 u_0$,
i.e.,
\ben
f & = & x (1+ \frac{2 u_0}{x^2})^{1/2} = \sum_{j=0}^\infty  \binom{1/2}{j} (2u_0)^j x^{1-2j} \\
& = & x +   \frac{u_0}{x}-\frac{1}{2} \frac{u_0^2}{x^3}
+\frac{1}{2} \frac{u_0^3}{x^5} - \frac{5}{8} \frac{u_0^4}{x^7}
+ \frac{7}{8} \frac{u_0^5}{x^9}- \frac{21}{16} \frac{u_0^6}{x^{11}} + \cdots,
\een
and let
\ben
x  & = & f (1- \frac{2u_0}{f^2})^{1/2}
= \sum_{j=0}^\infty  \binom{1/2}{j} (-2u_0)^j x^{1-2j} \\
& = & f - \frac{u_0}{f}-\frac{1}{2} \frac{u_0^2}{f^3}
- \frac{1}{2} \frac{u_0^3}{f^5} - \frac{5}{8} \frac{u_0^4}{f^7}
- \frac{7}{8} \frac{u_0^5}{f^9}- \frac{21}{16} \frac{u_0^6}{f^{11}}  - \cdots
\een
be its inverse series.
One can now note the meaning of the coefficients of $f^{2n+3}$ in $x$:

\begin{prop}
The following equalities hold:
\be
x|_{f^{2n+3}} = - \frac{(2n+1)!!}{(n+2)!} u_0^{n+2} =  -  \frac{\pd F_0}{\pd u_n}(u_0, 0, \dots).
\ee
\end{prop}

This is not just a nice coincidence
but  a special case of the mirror symmetry between the intersection theory of $\psi$-classes on $\Mbar_{g,n}$ and
the quantum deformation theory of the Airy curve to be presented below.

\subsection{Special deformation}

In last subsection we have seen that
the miniversal deformation of the Airy curve
$$y = \half x^2 + u_0$$
is given by the Puiseux series:
\be
x = f - \frac{u_0}{f} - \sum_{n \geq 0} \frac{\pd F_0}{\pd u_n}(u_0, 0, \dots) \cdot f^{-2n-3}.
\ee
This suggests  to include the variables $u_1, u_2, \dots$ in the deformation,
as in the following:

\begin{thm} \label{thm:Existence}
Consider the following series:
\be \label{eqn:X-in-F}
x = f - \sum_{n \geq 0} (2n+1) u_n f^{2n-1} - \sum_{n \geq 0} \frac{\pd F_0}{\pd u_n}(\bu) \cdot f^{-2n-3}.
\ee
Then one has
\be \label{eqn:x^2}
\begin{split}
x^2 = &
2 y \biggl(1- \sum_{n \geq 1} (2 n+1) u_n (2y)^{n-1}\biggr)^2 \\
- & 2 u_0 \biggl(1-  \sum_{n \geq 1} (2n+1) u_n (2y)^{n-1} \biggr) \\
+ & 2 \sum_{n \geq 0} \sum_{k \geq n+2} (2k+1) u_k \cdot \frac{\pd F_0}{\pd u_n} \cdot (2y)^{k-n-2} .
\end{split}
\ee
In particular,
\be
(x^2)_- = 0.
\ee
Here for a formal series $\sum_{n \in \bZ} a_n f^n$,
\be
(\sum_{n \in \bZ} a_n f^n)_+ = \sum_{n \geq 0} a_n f^n, \;\;\;\;
(\sum_{n \in \bZ} a_n f^n)_ - = \sum_{n < 0} a_n f^n.
\ee
\end{thm}

\begin{proof}
This is actually equivalent to the Virasoro constraints for $F_0$.
Indeed,
\ben
x^2 & = & \biggl( f - \sum_{n \geq 1} (2n+1) u_n f^{2n-1}
- \frac{u_0}{f} - \sum_{n \geq 0} \frac{\pd F_0}{\pd u_n} \cdot f^{-2n-3} \biggr)^2  \\
& = & \biggl( f- \sum_{n \geq 1} (2n+1) u_n f^{2n-1}\biggr)^2  \\
& - & 2 \biggl( f- \sum_{n \geq 1} (2n+1) u_n f^{2n-1}\biggr) \cdot
\biggl( \frac{u_0}{f} + \sum_{n \geq 0} \frac{\pd F_0}{\pd u_n} \cdot f^{-2n-3}\biggr) \\
& + & \biggl( \frac{u_0}{f} + \sum_{n \geq 0} \frac{\pd F_0}{\pd u_n} \cdot f^{-2n-3}\biggr)^2 \\
& = & f^2 \biggl(1- \sum_{n \geq 1} (2n+1) u_n f^{2n-2}\biggr)^2 - 2 u_0 \biggl( 1 - \sum_{n \geq 1} (2n+1) u_n f^{2n-2}\biggr)  \\
& - & 2 \sum_{n \geq 0} \frac{\pd F_0}{\pd u_n} \cdot f^{-2n-2} + 2 \sum_{n \geq 1} (2n+1) u_n f^{2n-1} \cdot
 \sum_{n \geq 0} \frac{\pd F_0}{\pd u_n} \cdot f^{-2n-3}  \\
& + & \biggl( \frac{u_0}{f} + \sum_{n \geq 0} \frac{\pd F_0}{\pd u_n} \cdot f^{-2n-3}\biggr)^2.
\een
It follows that
\ben
(x^2)_- & = & 2 \biggl( \frac{t_0^2}{2}   \cdot f^{-2} - \sum_{n \geq 0} \frac{\pd F_0}{\pd u_n} \cdot f^{-2n-2} \\
& + & \sum_{n \geq 0} \sum_k (2k+1) u_k \cdot  \frac{\pd F_0}{\pd u_{k+n-1}} \cdot f^{-2n-2} \\
& + & \half \sum_{n \geq 2} \sum_{j+k = n-2} \frac{\pd F_0}{\pd u_j}
\cdot \frac{\pd F_0}{\pd u_k} \cdot f^{-2n-2} \biggr).
\een
It vanishes by Virasoro constraints for $F_0$.
It follows that
\ben
x^2 & = & (x^2)_+
=  f^2 \biggl(1- \sum_{n \geq 1} (2n+1) u_n f^{2n-2}\biggr)^2 \\
& - & 2u_0 \biggl( 1 - \sum_{n \geq 1} (2n+1) u_n f^{2n-2}\biggr)  \\
& + &  2 \sum_{n \geq 0} \sum_{k \geq n+2} (2k+1) u_k \cdot \frac{\pd F_0}{\pd u_n} \cdot f^{2(k-n-2)} .
\een
The proof is completed by recalling $2y=f^2$.
\end{proof}

Recall we have assigned that $\deg t_n = 2-2n$,
and because $t_n = (2n-1)!! \cdot u_n$,
we have $\deg u_n = 2 -2n$.
Also because $\deg F_0 = 6$,
so we have
$$
\deg \frac{\pd F_0}{\pd u_n} = 2n+4.
$$
Therefore,
if we assign
\be
\deg f = 1,
\ee
then the right-hand side of \eqref{eqn:X-in-F} is weighted homogeneous of degree $1$,
it follows that we should take
\be
\deg x = 1.
\ee

\subsection{Uniqueness of special deformation of the Airy curve}

Let us first prove a simple combinatorial result.

\begin{thm} \label{thm:Uniqueness}
There exists a unique series
\be
x = f - \sum_{n \geq 0} v_n f^{2n-1} - \sum_{n \geq 0} w_n f^{-2n-3}
\ee
such that
each $w_n \in \bC[[v_0, v_1, \dots]]$ and
\be \label{eqn:x2-=0}
(x^2)_- = 0.
\ee
\end{thm}

\begin{proof}
We begin by rewriting \eqref{eqn:x2-=0} as a sequence of equations:
\bea
&& w_0 =  \frac{1}{2} v_0^2 + v_1 w_0 + v_2 w_1 + v_3 w_2 + \cdots, \label{eqn:Rec-w0} \\
&& w_1 = \;\;\;\; \;\;\;\;\;\; v_0 w_0 + v_1 w_1 + v_2w_2 + \cdots, \label{eqn:Rec-w1} \\
&& w_2 = \;\;\;\; \;\;\;\;\;\;\;\;\;\;\;\;\;\;\;\;\;\; v_0 w_1 + v_1 w_2 + \cdots + \half w_0^2, \label{eqn:Rec-w2} \\
&& w_3 = \;\;\;\; \;\;\;\;\;\;\;\;\;\;\;\;\;\;\;\;\;\; \;\;\;\;\;\;\;\;\;\;\; v_0 w_2  + \cdots + w_0w_1, \label{eqn:Rec-w3} \\
&& \cdots\cdots\cdots\cdots \cdots
\eea
Write
$$
w_n = w_n^{(0)} + w_n^{(1)} + \cdots,
$$
where each $w^{(k)}_n$ consists of monomials in $v_0, v_1, \dots$ of ordinary degree $k$.
Using such decompositions,
one can deduce by induction from the above system of equations:
\be
w_n^{(j)} = 0, \;\;\; n \geq 0, \;\; j =0, \dots, n+1,
\ee
and furthermore,
\ben
&& w_0^{(2)} = \half v_0^2, \\
&& w_0^{(n)} = \sum_{j \geq 1} v_j w_{j-1}^{(n-1)}, \;\;\; n \geq 3, \\
&& w_1^{(n)} = \sum_{j=0}^\infty v_j w_j^{(n-1)}, \;\;\; n \geq 3, \\
&& w_m^{(n)} = \sum_{j=0}^\infty v_j w_{j+m-1}^{n-1}
+ \half \sum_{j=0}^{m-2} \sum_{k=j}^{n-m+j} w_j^{(k)} w_{m-2-j}^{(n-k)}, \; m \geq 1, \; n \geq m+2.
\een
It follows that one can recursively determine all $w_m^{(n)}$ from the initial value $w_0^{(2)} = \half v_0^2$.
\end{proof}

By combining Theorem \ref{thm:Existence} with Theorem \ref{thm:Uniqueness},
we then get:

\begin{thm}
The equation
\be
(x^2)_- = 0
\ee
for a series
\be
x = f - \sum_{n \geq 0} (2n+1) u_n f^{2n-1} - \sum_{n \geq 0} w_n f^{-2n-3},
\ee
where each $w_n \in \bC[[u_0, u_1, \dots]]$ has a unique solution given by:
\be
x = f - \sum_{n \geq 0} (2n+1) u_n f^{2n-1} - \sum_{n \geq 0} \frac{\pd F_0}{\pd u_n}(\bu) \cdot f^{-2n-3},
\ee
as a series in $\bu = (u_0, u_1, \dots)$,
where $F_0(\bu)$ is the free energy of the 2D topological gravity in genus zero.
\end{thm}

\subsection{Deformation of the superpotential function}

From \eqref{eqn:x^2} one can also derive formula for deformation of the superpotential function.
By \eqref{eqn:x^2},
one can formally write:
\be
x^2 = a_0 + a_1 f^2 + a_2 f^4 + \cdots,
\ee
where
\bea
&& a_0 = - 2 u_0(1-3u_1) + 2 \sum_{n \geq 0} (2n+5) u_{n+2} \frac{\pd F_0}{\pd u_n}, \\
&& a_1 = (1-3u_1)^2 + 2u_0 \cdot 5u_2 + 2\sum_{n \geq 0} (2n+7)u_{n+3} \cdot \frac{\pd F_0}{\pd u_n},
\eea
and for $m \geq 2$,
\be
\begin{split}
a_m = & -2 (1-3u_1) \cdot (2m+1)u_m + 2u_0 \cdot (2m+3) u_{m+1} \\
+ & \sum_{\substack{m_1, m_2 \geq 1 \\ m_1+m_2 = m+1}} (2m_1+1)u_{m_1} \cdot (2m_2+1)u_{m_2} \\
+ & 2 \sum_{n \geq 0} (2n+2m+1)u_{n+2+m} \cdot \frac{\pd F_0}{\pd u_n}.
\end{split}
\ee
Note each $a_m$ is weighted homogenous of degree
\be
\deg a_m = 2 - 2m.
\ee
For later use,
we note the following specialization of these coefficients:
\bea
&& a_0(u_0, u_1, 0, \dots) = - 2u_0 (1-3u_1), \\
&& a_1(u_0, u_1, 0, \dots) = (1-3u_1)^2, \\
&& a_n(u_0, u_1, 0, \dots) = 0,  \;\; \; n \geq 2.
\eea
We also note that by \eqref{eqn:Pd-un-F0},
for $m \geq 0$,
\be \label{eqn:Pd-un-am}
\frac{\pd a_m}{\pd u_n}(u_0) = 2(2n+1) \cdot \frac{(2n-2m-3)!!}{(n-m)!} u_0^{n-m},
\ee
where we have used the following convention:
\be
(-1)!! = 1, \;\; (-3)!! = -1, \;\; (-2n-1)!! = 0, \;\; n \geq 2.
\ee

Now we have
\be
\frac{x^2-a_0}{a_1} = f^2 + \frac{a_2}{a_1} f^4 + \frac{a_3}{a_1} f^6 + \cdots,
\ee
so we can apply the Lagrangian inversion to get:
\bea
f^2 & = & \frac{x^2-a_0}{a_1} + b_2 \cdot \biggl( \frac{x^2-a_0}{a_1} \biggr)^2 + \cdots \\
& = & c_{0} + c_1 x^2 + c_2 x^4 + c_3 x^6 + \cdots,
\eea
where $\{\frac{a_2}{a_1}, \frac{a_3}{a_1}, \dots\}$ and
$\{b_2, b_3, \dots\}$ are Lagrangian dual to each other,
and
\ben
&& c_0 = - \frac{a_0}{a_1} + \sum_{n=2}^\infty (-1)^nb_n \frac{a_0^n}{a_1^n}, \\
&& c_1 = \frac{1}{a_1^2} + \sum_{n=2}^\infty (-1)^{n-1} n b_n \frac{a_0^{n-1}}{a_1^n}, \\
&& c_m = \sum_{n = m}^\infty (-1)^m \binom{n}{m} b_n \frac{a_0^{n-m}}{a_1^n}, \;\; m \geq 2.
\een
We have already seen that
\bea
&& c_0(u_0, u_1, 0, \dots) = \frac{2u_0}{1-3u_1}, \\
&& c_1(u_0, u_1, 0, \dots) = \frac{1}{(1-3u_1)^2}, \\
&& c_m(u_0, u_1, 0, \dots) = 0, \;\; m \geq 2.
\eea

The following are some examples:
When $u_j = 0$ for $j \geq 1$,
the equation of the Airy curve is deformed to:
\be
x^2 = 2y - 2u_0,
\ee
and the superpotential function is deformed to
\be \label{eqn:y=xu0}
y = \frac{1}{2}x^2 + u_0.
\ee
When $t_j = 0$ for $j \geq 2$,
the curve is deformed to
\be
x^2 = 2 (1- 3 u_1)^2 y - 2 u_0(1 - 3 u_1),
\ee
and the superpotential is deformed to
\be \label{eqn:y=xu0u1}
y = \frac{1}{2(1-3u_1)^2} x^2 + \frac{u_0}{1-3u_1}.
\ee
In other words,
by making the following transformation:
\begin{align*}
x & \mapsto \frac{x}{1-3u_1}, & u_0 & \mapsto \frac{u_0}{1-3u_1},
\end{align*}
one can obtain \eqref{eqn:y=xu0u1} from \eqref{eqn:y=xu0}.
When $t_j = 0$ for $j \geq 3$,
the curve is deformed to:
\ben
x^2 & = & -2u_0(1-3u_1) +  10 u_2 \frac{\pd F_0}{\pd u_0}(u_0, u_1, u_2) \\
& + & ( (1-3u_1)^2+10u_0 u_2) (2y) \\
& + & (30u_1u_2-10u_2 (1-3u_1)) (2y)^2 +25u_2^2 (2y)^3,
\een
where $F_0$ is given by \eqref{eqn:F0(t0,t1,t2)}.
By making the transformation
\begin{align*}
x & \mapsto \frac{x}{1-3u_1}, & u_0 & \mapsto \frac{u_0}{1-3u_1}, & u_2 & \mapsto \frac{u_2}{1-3u_1}
\end{align*}
one can reduce to the case of $u_1 = 0$:
\be \label{eqn:Deformation-u0u2}
\begin{split}
x^2  = & \frac{2}{135u_2}((1-30u_0u_2)^{3/2} - 1 - 90 u_0u_2) \\
+ & (1+10u_0u_2) \cdot f \\
- & 10 u_2 \cdot f^2 + 25 u_2^2 \cdot f^3.
\end{split}
\ee
One can solve this equation explicitly by writing $f$ as a function of $x^2$.
To make the notations simpler,
write $C=5u_2f$, $V = 10 u_0u_2$,
$U=(1-3V)^{1/2}$ and $X = 5u_2x^2$.
Then one can rewrite the above equation as follows:
\be
C^3 - 2C^2 - \frac{1}{3}(U+2)(U-2)C + \frac{2}{27}(U+2)^2(U-1) - X = 0.
\ee
An  explicit solution is given by:
\ben
C & = & \frac{1}{6} [108X-8U^3+12(81X^2-12XU^3)^{1/2}]^{1/3} \\
& + & \frac{2U^2}{3}
[108 X-8 U^3+12(81 X^2-12X U^3)^{1/2}]^{-1/3} + \frac{2}{3}.
\een
However, this is not very useful if we want to expand it as a series in $X$.
Write
$C = \sum_{n=0}^\infty C_n X^n=C_0+Y$.
Then one has
\be
C_0 = \frac{2}{3} (1- U),
\ee
and
\be
Y^3-2UY^2+U^2Y =X,
\ee
After dividing both sides by $U^3$:
\be
\biggl(\frac{Y}{U} \biggr)^3 - 2\biggl(\frac{Y}{U}\biggr)^2 + \frac{Y}{U} = \frac{X}{U^3}.
\ee
This can be solved explicitly by the following series of generalized Catalan numbers:
\be
\frac{Y}{U} = \sum_{n=0}^\infty \frac{1}{n+1} \binom{3n+1}{n} \biggl(\frac{X}{U^3} \biggr)^{n+1}.
\ee
Therefore,
\ben
C & = & \frac{2}{3} (1- U) + \sum_{n =0}^\infty \frac{1}{n+1} \binom{3n+1}{n} \frac{X^{n+1}}{U^{3n+2}} \\
& = & \frac{2}{3} (1- U) + \frac{X}{U^2} + \frac{2X^2}{U^5} + \frac{7X^3}{U^8}
+ \frac{30X^4}{U^{11}} + \frac{143X^5}{U^{14}} + \cdots.
\een
It follows that
\be \label{eqn:f-u0-u2}
\begin{split}
f = & \frac{2}{15} \frac{1- (1-30u_0u_2)^{1/2}}{u_2} \\
+ & \sum_{n =0}^\infty \frac{1}{n+1} \binom{3n+1}{n}
\frac{(5u_2)^nx^{2n+2}}{(1-30u_0u_2)^{(3n+2)/2}}.
\end{split}
\ee
More explicitly,
\ben
2y & = & (2u_0+15u_2u_0^2+225u_2^2u_0^3 + \frac{16875}{4} u_2^3u_0^4 + \frac{354375}{4} u_2^4u_0^5 + \cdots) \\
& + & (1+30u_2u_0+900u_2^2u_0^2+27000u_2^3u_0^3+810000u_2^4u_0^4 +\cdots) x^2 \\
& + & (10u_2+750u_2^2u_0+39375u_2^3u_0^2+1771875u_2^4u_0^3+\cdots) x^4 \\
& + & (175u_2^2+21000u_2^3u_0+1575000u_2^4u_0^2 + \cdots) x^6 \\
& + & \cdots.
\een
By letting $u_2 = 0$,
one recovers the equation $y = \half x^2 + u_0$.

This example reveals some interesting features of the deformation \eqref{eqn:x^2}.
First,
by keeping $u_n = 0$ for $n >2$ and letting  $u_2 \neq 0$,
the deformation changes  the Airy curve from a rational curve to a hyperelliptic curve of degree $6$,
hence the genus is changed from $g=0$ to $g=3$.
In general,
one can do this recursively by taking one more of $u_n$ to be nonzero,
the effect will be including terms of higher degrees in $y$,
and so the curve will have larger genus.
Secondly,
the constant term of the polynomial in $y$ on the right-hand side of  \eqref{eqn:Deformation-u0u2}
is no longer a polynomial in $u_0, u_1, u_2$.
Thirdly,
the superpotential $\frac{1}{2} f$  is deformed to become an infinite formal power series in $x^2$.
There does not seem to have motivations to study such deformations from a traditional mathematical point of view.
So we have another nice example of how string theory enriches mathematical researches.

\section{Landau-Ginzburg Theory Associated with the Miniversal Deformation of the Airy Curve}

\label{sec:Landau-Ginzburg}

In last section we have seen that the miniversal deformation
$$y = \half x^2 + t_0$$
of the Airy curve $y = \half x^2$ encodes the information of $\frac{\pd F_0}{\pd u_n}(u_0)$.
In this section we take $W = y = \frac{1}{2} x^2 + t_0$ as the superpotential and consider the corresponding
Landau-Ginzburg theory by defining some correlation functions on the small phase space in this theory.
It is clear that one does not get a field theory in the sense of Section \ref{sec:Prelim},
nevertheless,
we will identify the correlation functions we define in this section with the corresponding correlation functions
in the theory of the 2D topological gravity on the small phase space.
The results in this section motivates the extension to the big phase space in the next section.

\subsection{The primary field and its descendant fields}

The primary field $\phi_0$ is defined by:
\be
\phi_0(x) : = \frac{\pd W}{\pd u_0} = (\pd_x L)_+ = 1,
\ee
where $L= (2W)^{1/2}$ is the Puiseux series:
\be
L = (x^2+ 2u_0)^{1/2} = x (1+\frac{2u_0}{x^2})^{1/2}
= x+ \frac{u_0}{x} - \frac{1}{2} \frac{u_0^2}{x^3} + \half \frac{u_0^3}{x^5} + \cdots.
\ee
It gives a basis of the Jacobian ring
\be
J_W := \bC[x]/\corr{\pd_x W}.
\ee
Inspired by \cite{Losev},
the authors of \cite{E-K-Y-Y} constructed
the $n$-th gravitational descendent field of $\phi_0(x)$ as follows:
\be \label{def:Descendants}
\sigma_n(\phi_0) = \frac{1}{(2n-1)!!} (L^{2n} \pd_xL)_+
= \frac{1}{(2n+1)!!} (\pd_x L^{2n+1})_+.
\ee

The following are some explicit examples:
\ben
&& \phi_0 = 1, \\
&& \sigma_1(\phi_0) = x^2 + u_0, \\
&& \sigma_2(\phi_0) = \frac{1}{3!!} (x^4 + 3u_0 x^2 + \frac{3}{2} u_0^2), \\
&& \sigma_3(\phi_0) = \frac{1}{5!!} (x^6 + 5u_0x^4+\frac{15}{2}u_0^2x^2 + \frac{5}{2}u_0^3), \\
&& \sigma_4(\phi_0) = \frac{1}{7!!} (x^8 + 7u_0x^6+ \frac{35}{2}u_0^2x^4+\frac{35}{2}u_0^3x^2 + \frac{35}{8} u_0^4), \\
&& \sigma_5(\phi_0) = \frac{1}{9!!} (x^{10} +9 u_0x^8+ \frac{63}{2}u_0^2x^6+ \frac{105}{2}u_0^3x^4+\frac{315}{8} u_0^4x^2 + \frac{63}{8} u_0^5).
\een
Since they are polynomials in $x$,
one can understand them as fields living on the Airy curve $y = \frac{1}{2}x^2$,
or maybe more appropriately,
on the deformed Airy curve $y = \half x^2 + u_0$.

\subsection{Variations of the descendant fields with respect to $u_0$}

It is clear that
\bea
&& \pd_{u_0} W = \phi_0(x) = 1, \\
&& \pd_{u_0} \phi_0(x) = 0.
\eea
By differentiating both sides of $L^2 = W$,
one gets:
\be \label{eqn:Pd-u0-L}
\pd_{u_0} L = \frac{1}{L}.
\ee

\begin{lem}
The variation of $\sigma_n(\phi_0)$ is given by:
\be
\pd_{u_0} \sigma_n(\phi_0) = \sigma_{n-1}(\phi_0),
\ee
for $n \geq 1$.
\end{lem}

\begin{proof}
Take $\pd_{u_0}$ on both sides of \eqref{def:Descendants} and apply \eqref{eqn:Pd-u0-L}:
\ben
&& \pd_{u_0} \sigma_n(\phi_0) \\
& = & \frac{2n}{(2n-1)!!} (L^{2n-1} \cdot \pd_{u_0}L \cdot \pd_x L)_+
+ \frac{1}{(2n-1)!!} (L^{n} \pd_x (\pd_{u_0} L))_+ \\
& = &  \frac{2n}{(2n-1)!!} (L^{2n-1} \cdot \frac{1}{L} \cdot \pd_x L)_+
+ \frac{1}{(2n-1)!!} (L^{n} \pd_x (\frac{1}{L}))_+ \\
& = & \frac{1}{(2n-3)!!} (L^{2n-2} \pd_x L)_+ = \sigma_{n-1}(\phi_0).
\een
\end{proof}

Because $\sigma_n(\phi_0)$ is a degree $n$ polynomial in $x^2$,
with leading term $\frac{x^{2n}}{(2n-1)!!}$,
as a corollary,
we have

\begin{cor}
The following recursion relation holds:
\be \label{eqn:Rec-sigma(phi0)}
\sigma_n(\phi_0) = \frac{x^{2n}}{(2n-1)!!} + \int_0^{u_0} \sigma_{n-1}(\phi_0) du_0.
\ee
\end{cor}

\subsection{Explicit expressions for descendant fields}

By induction one then gets:

\begin{prop}
An explicit solution of \eqref{eqn:Rec-sigma(phi0)} is given by:
\be \label{eqn:Sigma(phi0)}
\sigma_n(\phi_0) =  \sum_{j=0}^n \frac{1}{(2j-1)!!(n-j)!} x^{2j} u_0^{n-j}
\ee
\end{prop}

One can use \eqref{eqn:Sigma(phi0)} to get the following formula for the generating series of $\{\sigma_n(\phi_0)\}_{n \geq 0}$
($\sigma_0(\phi_0) = \phi_0$):
\be
\sum_{n \geq 0} t^n \sigma_n(\phi_0)
= e^{u_0t} (1 + \sqrt{\pi x^2t/2} e^{x^2t/2} \erf(\sqrt{x^2t/2}) ),
\ee
where $\erf(x)$ is the error function:
\be
\erf(x) = \frac{2\int_0^x e^{-t^2} dt} {\sqrt{\pi}}.
\ee

As another application of \eqref{eqn:Sigma(phi0)},
we have

\begin{prop}
The following recursion relation holds for $n \geq 0$:
\be \label{eqn:Sigma1-Sigman}
\begin{split}
\sigma_1(\phi_0) \cdot & (2n-1)!!\sigma_n(\phi_0)
= (2n+1)!! \sigma_{n+1}(\phi_0) \\
- & u_0 \cdot (2n-1)!! \sigma_n(\phi) + \frac{(2n-1)!!}{(n+1)!} u_0^{n+1}.
\end{split}
\ee
\end{prop}

Using generating series,
the recursion relations \eqref{eqn:Sigma1-Sigman} can be explicitly solved as follows:
\be
\sum_{n \geq 0} (2n-1)!!\sigma_n(\phi_0) t^n
= \frac{\sqrt{1-2u_0t}}{1-t(x^2+2u_0)}.
\ee
It follows that
\be \label{eqn:Sigman(phi0)-Explicit2}
(2n-1)!!\sigma_n(\phi_0)
= \sum_{j=0}^n \frac{(2j-3)!!}{j!} \cdot u_0^j (x^2+2u_0)^{n-j}.
\ee

\subsection{Descendant operator algebra and structure constants}

It is easy to see that
\be
(2j-1)!! \sigma_j(\phi_0) \cdot (2k-1)!! \sigma_k(\phi_0)
= \sum_{j=0}^{j+k} c_{jk}^l u_0^{j+k-l} \cdot (2l-1)!! \sigma_l(\phi_0)
\ee
for some constants $c_{jk}^l$.
In other words,
$\{\sigma_n(\phi_0)\}_{n \geq 0}$ generate an algebra over the ring $\bC[u_0]$.
We will call this algebra the {\em descendant algebra on the small phase}.

To determine the structure constants $c_{jk}^lu_0^{j+k-l}$,
let
\be
f_n(x) = \sum_{j=0}^n \frac{(2n-1)!!}{(2j-1)!!(n-j)!} x^{2j},
\ee
then
\be
f_j(x) \cdot f_k(x) = \sum_{l=0}^{j+k} c_{jl}^l f_l(x).
\ee
For example,
\ben
&& f_0(x) = 1, \\
&& f_1(x) = x^2 + 1, \\
&& f_2(x) = x^4 + 3 x^2 + \frac{3}{2}, \\
&& f_3(x) = x^6 + 5 x^4+\frac{15}{2} x^2 + \frac{5}{2} , \\
&& f_4(x) = x^8 + 7 x^6+ \frac{35}{2} x^4+\frac{35}{2} x^2 + \frac{35}{8}, \\
&& f_5(x) = x^{10} +9 x^8+ \frac{63}{2} x^6+ \frac{105}{2} x^4+\frac{315}{8} x^2 + \frac{63}{8}.
\een
By \eqref{eqn:Sigma1-Sigman},
for $n \geq 1$,
\be \label{eqn:f1fn}
f_1 f_n = f_{n+1} - f_n + \frac{(2n-1)!!}{(n+1)!}.
\ee
One can check that
\ben
&& f_2 f_2 = (f_4 - f_3 - \frac{1}{2} f_2) + (\frac{1}{2} f_1 + \frac{5}{8} f_0), \\
&& f_2 f_3 = (f_5 - f_4 - \frac{1}{2} f_3) + (\frac{5}{8} f_1 + \frac{7}{8} f_0), \\
&& f_2 f_4 = (f_6 - f_5 - \frac{1}{2} f_4) + (\frac{7}{8} f_2 + \frac{21}{16} f_0), \\
&& f_3 f_3 = (f_6 - f_5 - \frac{1}{2} f_4 - \frac{1}{2} f_3) + (\frac{5}{8} f_2 + \frac{7}{8} f_1 + \frac{21}{16} f_0).
\een
One recognizes  the structure constants as the coefficients of the following series:
\be
\begin{split}
\sqrt{1-2x} = & 1- \sum_{n=1}^\infty \frac{(2n-3)!!}{n!} x^n \\
= & 1 - x - \frac{1}{2}x^2 - \frac{1}{2} x^3 - \frac{5}{8} x^4 - \frac{7}{8} x^5 - \frac{21}{16} x^6 - \cdots,
\end{split}
\ee

\begin{prop}
For $k \geq j \geq 1$,
\be \label{eqn:fjfk}
\begin{split}
f_j f_k = & f_{j+k} - \sum_{l=1}^j \frac{(2l-3)!!}{l!} f_{j+k-l} \\
+ & \sum_{l=1}^j  \frac{(2(k+l)-3)!!}{(k+l)!} f_{j-l}.
\end{split}
\ee
\end{prop}

\begin{proof}
We prove \eqref{eqn:fjfk} by induction on $j$.
For $j=1$,
\eqref{eqn:fjfk} is just \eqref{eqn:f1fn}.
Suppose that \eqref{eqn:fjfk} holds for some $j \geq 1$.
Then by \eqref{eqn:f1fn},
\ben
f_{j+1} = f_1f_j + f_j  - \frac{(2j-1)!!}{(j+1)!}.
\een
And so for $k \geq j+1$,
\ben
f_{j+1}f_k & = & (f_1f_j + f_j - \frac{(2j-1)!!}{(j+1)!}) \cdot f_k \\
& = & f_1 \cdot \biggl(f_{j+k} - \sum_{l=1}^j \frac{(2l-3)!!}{l!} f_{j+k-l}
+ \sum_{l=1}^j  \frac{(2(k+l)-3)!!}{(k+l)!} f_{j-l} \biggr) \\
& + & \biggl( f_{j+k} - \sum_{l=1}^j \frac{(2l-3)!!}{l!} f_{j+k-l}
+ \sum_{l=1}^j  \frac{(2(k+l)-3)!!}{(k+l)!} f_{j-l} \biggr) \\
& - &  \frac{(2j-1)!!}{(j+1)!} f_k \\
& = & \biggl( f_{j+k+1} - f_{j+k} + \frac{(2(j+k)-1)!!}{(j+k+1)!} \biggr) \\
&- & \sum_{l=1}^j \frac{(2l-3)!!}{l!} \biggl( f_{j+k-l+1} - f_{j+k-l} + \frac{(2(j+k-l)-1)!!}{(j+k-l+1)!} \biggr) \\
& + &  \sum_{l=1}^j  \frac{(2(k+l)-3)!!}{(k+l)!} \biggl( f_{j-l+1} - f_{j-l} + \frac{(2(j-l)-1)!!}{(j-l+1)!} \biggr) \\
& + & \biggl( f_{j+k} - \sum_{l=1}^j \frac{(2l-3)!!}{l!} f_{j+k-l}
+ \sum_{l=1}^j  \frac{(2(k+l)-3)!!}{(k+l)!} f_{j-l} \biggr) \\
& - &  \frac{(2j-1)!!}{(j+1)!} f_k.
\een
Here in the last equality we have used \eqref{eqn:f1fn}.
The proof is completed by obvious cancelations.
\end{proof}

\subsection{Some $n$-point correlation functions on the small phase space}

In the physics literature \cite{DVV1},
the authors first derived based on physical arguments  a formula for three-point function   on the small phase space
in the Landau-Ginzburg model associated with $W$,
then obtained formulas for two-point functions and one-point functions by integrations.
The free energy was obtained from the one-point functions in another work of the same authors \cite{DVV2}
using the weighted homogeneity of the free energy.
General $n$-point functions ($n \geq 4$) can be obtained by differentiating the three-point function.
For generalizations that include gravitational descendants and extension to big phase space,
we use \cite{E-K-Y-Y, E-Y-Y, Mukherjee} as references.
In this subsection we will present a mathematical reformulation of \cite{DVV1, E-K-Y-Y} by reversing the above procedure.
We will take the formula for one-point function as the definition of the one-point function,
and define $n$-point functions ($n \geq 2$) by taking derivatives.

We define the one-point function on the small phase space by:
\be \label{eqn:One-Point}
\corrr{\phi_0(x)}_0(u_0)
=  \frac{1}{3} \res_{x=\infty} ( L^3 dx)
= \frac{1}{3} \res ((2W)^{3/2} dx),
\ee
and so for $n \geq 2$,
the $n$-point function are defined by:
\be
\corrr{ \phi_0(x)^n}_0(u_0) = \frac{\pd}{\pd u_0} \corr{ \phi_0(x)^{n-1}}_0(u_0).
\ee
Using the equation \eqref{eqn:Pd-u0-L},
one can check that for $n \geq 2$,
\be \label{eqn:Cor-Phi0n}
\begin{split}
& \corrr{ \phi_0(x)^n}_0(u_0) \\
= & \begin{cases}
\res (L dx ), & n = 2, \\
(-1)^{n-1}(2n-5)!! \res\biggl(\frac{1}{L^{2n-5}}dx \biggr) = 0, & n \geq 4.
\end{cases}
\end{split}
\ee

In \cite[(46)]{E-K-Y-Y},
the following generalization of \eqref{eqn:One-Point} was given:
\be
\corrr{\sigma_n(\phi_0)}_0(u_0)  = \frac{1}{(2n+3)!!} \res(L^{2n+3}dx).
\ee
We will take this as the definition of $\corrr{\sigma_n(\phi_0)}_0(u_0)$,
and so for $m \geq 1$, define
\be
\corrr{\sigma_n(\phi_0) \phi_0(x)^m}_0(u_0) = \frac{\pd}{\pd u_0} \corrr{\sigma_n(\phi_0) \phi_0(x)^{m-1}}_0(u_0).
\ee
One can easily see that
\be
\corrr{\sigma_n(\phi_0) \phi_0(x)^m}_0(u_0)
=  \frac{\prod_{j=0}^{m-1} (2n+3-2j)}{(2n+3)!!} \res(L^{2n-2m+3}dx).
\ee
In particular,
when $n=0$,
\be
\corrr{ \phi_0(x)^{m+1}}_0(u_0)
=  \frac{\prod_{j=0}^{m-1} (3-2j)}{3} \res(L^{-2m+3}dx).
\ee
This matches with \eqref{eqn:Cor-Phi0n}.
Also when  $m=3$,
\ben
&& \corr{\sigma_n(\phi_0) \phi_0^2}_0(u_0) =  \frac{1}{(2n-1)!!} \res (L^{2n-1} dx)
= \res\biggl( \frac{\sigma_n(\phi_0) \phi_0\phi_0}{\pd_x W} dx \biggr).
\een
Inspired by the last equality and \cite{DVV1},
we define
\be \label{Def:3-Point}
\begin{split}
& \corrr{ \sigma_{n_1}(\phi_0) \sigma_{n_2}(\phi_0) \sigma_{n_3}(\phi_0) }_0(u_0) \\
= & \res \biggl( \frac{\sigma_{n_1}(\phi_0) \sigma_{n_2}(\phi_0) \sigma_{n_3}(\phi_0) }{\pd_x W} dx \biggr),
\end{split}
\ee
and for $m \geq 1$ recursively define:
\be
\begin{split}
& \corrr{ \sigma_{n_1}(\phi_0) \sigma_{n_2}(\phi_0) \sigma_{n_3}(\phi_0) \phi_0^m }_0(u_0) \\
= & \pd_{u_0}
\corr{ \sigma_{n_1}(\phi_0) \sigma_{n_2}(\phi_0) \sigma_{n_3}(\phi_0) \phi_0^{m-1} }_0(u_0).
\end{split}
\ee

\subsection{Compatibility of the definitions}

In last subsection,
we have given the definitions of the correlation functions $ \corrr{\sigma_n(\phi_0) \phi_0(x)^m}_0(u_0)$ and
 $\corrr{ \sigma_{n_1}(\phi_0) \sigma_{n_2}(\phi_0) \sigma_{n_3}(\phi_0) \phi_0^m }_0(u_0)$.
We now check their compatibility.
First we need the following:

\begin{prop}
The following formulas hold:
\bea
&& \corrr{ \sigma_{n}(\phi_0)  }_0(u_0) = \frac{u_0^{n+2} }{(n+2)!}, \label{eqn:Explicit-2-Point} \\
&& \corrr{ \sigma_{n_1}(\phi_0) \sigma_{n_2}(\phi_0) \sigma_{n_3}(\phi_0) }_0(u_0)
= \frac{u_0^{n_1+n_2+n_3} }{\prod_{j=1}^3 n_j!}. \label{eqn:Explicit-3-Point}
\eea
\end{prop}

\begin{proof}
Formula \eqref{eqn:Explicit-3-Point} follows directly from the explicit formula for $\sigma_n(\phi_0)$ given by \eqref{eqn:Sigma(phi0)} and the definition
\eqref{Def:3-Point}.
For the proof of \eqref{eqn:Explicit-2-Point},
we proceed as follows:
\ben
&& \corrr{ \sigma_{n}(\phi_0)  }_0(u_0) \\
& = & \frac{1}{(2n+3)!!} \res( L^{2n+3} dx)  = - \frac{1}{(2n+3)!!} \res(x dL^{2n+3}) \\
& = & -\frac{1}{(2n+1)!!} \res (L^{2n+2} (L - \sum_{m=1}^\infty \frac{(2m-3)!!}{m!} \frac{1}{L^{2m-1}}) dL) \\
& = &  \frac{u_0^{n+2} }{(n+2)!}.
\een
In the above we have uased:
\ben
x = (L^2 - 2u_0)^{1/2} = L (1- \frac{2u_0}{L^2})^{1/2}
= L - \sum_{m=1}^\infty \frac{(2m-3)!!}{m!} \frac{1}{L^{2m-1}},
\een
\end{proof}

By taking derivatives on both sides of \eqref{eqn:Explicit-2-Point},
\ben
&& \corrr{ \sigma_{n}(\phi_0)  \phi_0 }_0(u_0) = \frac{u_0^{n+1} }{(n+1)!}, \\
&& \corrr{ \sigma_{n}(\phi_0)  \phi_0^2 }_0(u_0) = \frac{u_0^{n} }{n!}.
\een
On the other hand,
taking $n_1= n$ and $n_2 = n_3 = 0$ in \eqref{eqn:Explicit-3-Point},
\ben
&& \corrr{ \sigma_{n}(\phi_0)  \phi_0^2 }_0(u_0) = \frac{u_0^n }{n!}.
\een
It is a match.
One also has for $n_3=0$ in \eqref{eqn:Explicit-3-Point}:
\be
\corrr{ \sigma_{n_1}(\phi_0) \sigma_{n_2}(\phi_0) \phi_0 }_0(u_0)
= \frac{u_0^{n_1+n_2} }{n_1!n_2!},
\ee
and so after integration one would have obtained
\be
\corrr{ \sigma_{n_1}(\phi_0) \sigma_{n_2}(\phi_0)}_0(u_0)
= \frac{u_0^{n_1+n_2+1} }{(n_1+n_2+1) \cdot n_1!n_2!},
\ee
if the left-hand side were already defined.

\subsection{Identification with genus zero $n$-point functions in 2D topological gravity on the small phase space}

\begin{prop}
For $n_1, n_2, n_3, m \geq 0$,
\be
\corrr{ \sigma_{n_1}(\phi_0) \sigma_{n_2}(\phi_0) \sigma_{n_3}(\phi_0) \phi_0^m}_0(t_0)
= \frac{\pd^{m+3}F_0}{\pd t_{n_1} \pd t_{n_2} \pd t_{n_3} \pd t_0^m}(t_0).
\ee
\end{prop}

\begin{proof}
It suffices to prove the $m=0$ case.
Recall $u_{(0)} = \pd_{u_0}^2F_0$ satisfies the dispersionless KdV rierarchy \eqref{eqn:Dispersionless-Hierarchy}:
\ben
&& \pd_{t_n} u_{(0)} = \pd_{t_0} \frac{u_{(0)}^{n+1}}{(n+1)!}.
\een
Therefore, we have
\ben
&& \pd_{t_{n_1}} \pd_{t_{n_2}} \pd_{t_{n_3}} u_{(0)}
= \pd_{t_{n_1}} \pd_{t_{n_2}}\pd_{t_0} \frac{u_{(0)}^{n_3+1}}{(n_3+1)!} \\
& = & \pd_{t_{n_1}}\pd_{t_0} \biggl( \frac{u_{(0)}^{n_3}}{n_3!}  \pd_{t_{n_2}} u_{(0)} \biggr)
=  \pd_{t_{n_1}}\pd_{t_0} \biggl( \frac{u_{(0)}^{n_3} }{n_3!}  \pd_{t_0} \frac{u_{(0)}^{n_2+1}}{(n_2+1)!} \biggr) \\
& = & \pd_{t_{n_1}} \pd_{t_0}^2 \biggl( \frac{u_{(0)}^{n_2+n_3+1}}{n_2!n_3!(n_2+n_3+1)} \biggr)
= \pd_{t_0}^2 \biggl( \frac{u_{(0)}^{n_1+n_2+n_3}}{n_1!n_2!n_3!} \pd_{t_0} u_{(0)} \biggr).
\een
After integration with respect to $t_0$ twice,
one gets:
\ben
&&  \pd_{t_{n_1}} \pd_{t_{n_2}} \pd_{t_{n_3}} F_0
=  \frac{u_{(0)}^{n_1+n_2+n_3}}{n_1!n_2!n_3!} \pd_{t_0} u_{(0)}.
\een
Now we restrict to the small phase space and use the fact that
$u_{(0)}(t_0) = t_0 = u_0$ to get:
\ben
 \pd_{t_{n_1}} \pd_{t_{n_2}} \pd_{t_{n_3}} F_0 (t_0)
=  \frac{u_0^{n_1+n_2+n_3}}{n_1!n_2!n_3!}.
\een
The completed by comparing with \eqref{eqn:Explicit-3-Point}.
\end{proof}

\section{Field Theory Associated with Special Deformation of the Airy Curve}
\label{sec:Deformed-LG-Theory}

The discussion in last section clearly indicates that to obtain a field that is the mirror theory to the
theory of 2D topological gravity,
miniversal deformation of the Airy curve does not suffice.
In \S \ref{sec:Deformation} we have constructed a general deformation
\be
x = f - \sum_{n \geq 0} (2n+1) u_n f^{2n-1} - \sum_{n \geq 0} \frac{\pd F_0}{\pd u_n}(\bu) \cdot f^{-2n-3}.
\ee
and
$$ 2y = f^2 = c_{0} + c_1 x^2 + c_2 x^4 + c_3 x^6 + \cdots,$$
of the Airy curve $y = \half x^2$ associated with the genus zero free energy $F_0(\bu)$.
In this section we will take $W = y$ as the superpotential and construct a field theory
in genus zero by generalizing the discussions in last section.
In next section we will extend this field theory to arbitrary genus
and prove it is mirror to the theory of 2D topological gravity.

\subsection{The primary field and its descendant fields}

As in last section,
the primary field $\phi_0$ is defined by:
\be
\phi_0 : = \frac{\pd W}{\pd u_0}.
\ee
Since now the superpotential $W$ depends also on the coupling constants $u_n$,
we define in the same fashion the following fields:
\be \label{Def:Phi-n}
\phi_n : = \frac{\pd W}{\pd u_n}.
\ee
This generalizes the treatment  in \cite{DVV1} to the big phase space.
Recall that $W \in \bC[[u_0, u_1, \dots]]_2$ and $\deg u_n = 2 -2n$.
Therefore,
$\phi_n$ is weighted homogenous of degree $2n$,
i.e.
\be
\phi_n \in \bC[[x,u_0,u_1, \dots]].
\ee
In the following we will explain why $\phi_n$ can be regarded as gravitational descendants of $\phi_0$.

We also define another sequence of fields $\sigma_n$ as in last section by
\be
\sigma_n = (L^{2n} \pd_xL)_+
= \frac{1}{2n+1} (\pd_x L^{2n+1})_+.
\ee
We now examine the relationship between $\{\phi_n\}$ and $\{\sigma_n\}$.

Recall the following relationship between $x^2$ and $W$:
\be \label{eqn:x2-in-W}
x^2 = a_0 + a_1 \cdot (2W) + a_2\cdot (2W)^2 + \cdots,
\ee
where
\bea
&& a_0 = - 2 u_0(1-3u_1) + 2 \sum_{n \geq 0} (2n+5) u_{n+2} \frac{\pd F_0}{\pd u_n}, \\
&& a_1 = (1-3u_1)^2 + 2u_0 \cdot 5u_2 + 2\sum_{n \geq 0} (2n+7)u_{n+3} \cdot \frac{\pd F_0}{\pd u_n},
\eea
and for $m \geq 2$,
\be
\begin{split}
a_m = & -2 (1-3u_1) \cdot (2m+1)u_m + 2u_0 \cdot (2m+3) u_{m+1} \\
+ & \sum_{\substack{m_1, m_2 \geq 1 \\ m_1+m_2 = m+1}} (2m_1+1)u_{m_1} \cdot (2m_2+1)u_{m_2} \\
+ & 2 \sum_{n \geq 0} (2n+2m+1)u_{n+2+m} \cdot \frac{\pd F_0}{\pd u_n}.
\end{split}
\ee
In particular,
each $a_n \in \bC[[u_0, u_1, \dots]]_{2-2m}$,
the degree $2-2m$ part of $\bC[[u_0, u_1, \dots]]$,
while $W \in \bC[[x,u_0, u_1, \dots]]_2$.

\begin{lem}
The following formula holds:
\be \label{eqn:phin-in-Pd-W}
\phi_n = -  \frac{\pd_{u_n} a_0 + 2\pd_{u_n} a_1 \cdot W + 4 \pd_{u_n} a_2 \cdot W^2 + \cdots}{2x} \pd_x W.
\ee
\end{lem}

\begin{proof}
After taking $\pd_x$ on both sides of \eqref{eqn:x2-in-W} we get:
\be
\pd_x W  = \frac{x}{a_1 + 4a_2 W + 12 a_3 W^2 + \cdots}.
\ee
Similarly,
after taking $\pd_{u_n}$ on both sides of \eqref{eqn:x2-in-W} we get:
\be
\pd_{u_n} W = -  \frac{\pd_{u_n} a_0 + 2\pd_{u_n} a_1 \cdot W + 4 \pd_{u_n} a_2 \cdot W^2 + \cdots}
{2(a_1 + 4a_2 W + 12 a_3 W^2 + \cdots)}.
\ee
It follows that
\be
\pd_{u_n} W = -  \frac{\pd_{u_n} a_0 + 2\pd_{u_n} a_1 \cdot W + 4 \pd_{u_n} a_2 \cdot W^2 + \cdots}{2x} \pd_x W.
\ee
This completes the proof.
\end{proof}

\begin{prop}
When restricted to the small phase space we have
\be
\phi_n(x, u_0) = (2n+1)!! \cdot  \sigma_n(\phi_0)(x, u_0).
\ee
\end{prop}

\begin{proof}
When restricted to the small phase space,
we have
\ben
&& W = \frac{1}{2} x^2 + u_0,
\een
and by \eqref{eqn:Pd-un-am},
\ben
&& \frac{\pd a_m}{\pd u_n}(u_0) = 2(2n+1) \cdot \frac{(2n-2m-3)!!}{(n-m)!} u_0^{n-m},
\een
therefore,
by \eqref{eqn:phin-in-Pd-W},
\ben
\phi_n & = & -\half (\frac{\pd a_0}{\pd u_n}(u_0)+ \frac{\pd a_1}{\pd u_n}(u_0) \cdot (x^2 + 2u_0)
+ \cdots + \frac{\pd a_n}{\pd u_n}(u_0) \cdot (x^2 + 2u_0)^n ) \\
& = & -(2n+1) \cdot \biggl( \frac{(2n-3)!!}{n!}u_0^n + \frac{(2n-5)!!}{(n-1)!}u_0^{n-1} \cdot (x^2+2u_0) \\
& + & \cdots
+ u_0 \cdot (x^2+2u_0)^{n-1} - (x^2+2u_0)^n \biggr) \\
& = & (2n+1)!! \cdot \sigma_n(\phi_0).
\een
In the last equality on the right-hand side we have used \eqref{eqn:Sigman(phi0)-Explicit2}
\end{proof}

We have related $\phi_n(x,u_0)$ to $\sigma_n(\phi_0)(x, u_0)$ defined by \eqref{def:Descendants}.
Unfortunately this relation do not hold anymore on the big phase space,
as shown by the following example.
Recall by \eqref{eqn:f-u0-u2}, when only $u_0$ and $u_2$ are nonzero,
\be
\begin{split}
2W = L^2 = & \frac{2}{15} \frac{1- (1-30u_0u_2)^{1/2}}{u_2} \\
+ & \sum_{n =0}^\infty \frac{1}{n+1} \binom{3n+1}{n}
\frac{(5u_2)^nx^{2n+2}}{(1-30u_0u_2)^{(3n+2)/2}}.
\end{split}
\ee
It follows that
\ben
&& \phi_0(x,u_0, u_2)
=  (1-30u_0u_2)^{-1/2}+ \frac{3}{2} \sum_{n=0}^\infty \binom{3n+2}{n+1}
\frac{(5u_2x^2)^{n+1}}{(1-30u_0u_2)^{(3n+4)/2}}.
\een
Next we compute $L$.
First we rewrite the above equality as follows:
\ben
L^2 & = & \frac{x^2}{1-30u_0u_2} \biggl(1 +  \frac{2}{15} \frac{(1- (1-30u_0u_2)^{1/2})(1-30u_0u_2)}{u_2} x^{-2} \\
& + & \sum_{n = 1}^\infty \frac{1}{n+1} \binom{3n+1}{n}
\frac{(5u_2)^nx^{2n}}{(1-30u_0u_2)^{3n/2}} \biggr).
\een
Then we note
\ben
&& (x (1 + x^{-2}(a_0 + a_2x^4 + a_3 x^6 + \cdots))^{1/2})_+ \\
& = & x (1- \frac{1}{4}a_0a_2+\frac{3}{16}a_0^2a_3 -\frac{5}{32} a_0^3a_4
- \frac{15}{64} a_0^2 a_2^2 + \frac{35}{64}  a_0^3a_2a_3 \\
& - & \frac{105}{256}  a_0^3 a_2^3 - \frac{315}{512} a_0^4a_2a_4
+ \frac{3465}{2048} a_0^4a_2^2a_3-\frac{315}{1024} a_0^4a_3^2+ \cdots) + \cdots,
\een
it follows that
\ben
L_+ = \frac{x}{(1-30u_0u_2)^{1/2}} (1-5u_0u_2-75 (u_0u_2)^2-\frac{2875}{2} (u_0u_2)^3+ \cdots) + \cdots,
\een
and so
\ben
\sigma_0 & = & \pd_x L_+ \\
& = &  \frac{1}{(1-30u_0u_2)^{1/2}} (1-5u_0u_2-75 (u_0u_2)^2-\frac{2875}{2} (u_0u_2)^3+ \cdots) + \cdots,
\een
It is then clear that:
\be
\phi_0(u_0,u_2) \neq \sigma_0(u_0,u_2).
\ee

\subsection{Family of flat connections}
Inspired by the theory of Frobenius manfiolds \cite{Dubrovin},
we introduce the following operators ($\lambda \in \bC$):
\be
\nabla^\lambda_{\frac{\pd}{\pd t_i}} \phi_k := \pd_{t_i} \phi_k + \lambda \phi_i \cdot \phi_k.
\ee

\begin{prop}
The operators $\nabla^\lambda_{\frac{\pd}{\pd t_i}}$ define a family of flat connections,
i.e.,
\be
\nabla^\lambda_{\frac{\pd}{\pd t_i}} \nabla^\lambda_{\frac{\pd}{\pd t_j}}\phi_k
= \nabla^\lambda_{\frac{\pd}{\pd t_j}} \nabla^\lambda_{\frac{\pd}{\pd t_i}}\phi_k,
\ee
for $i,j,k \geq 0$ and all $\lambda \in \bC$.
\end{prop}

\begin{proof}
\ben
\nabla^\lambda_{\frac{\pd}{\pd t_i}} \nabla^\lambda_{\frac{\pd}{\pd t_j}} \phi_k
& = & (\pd_{t_i} + \lambda \phi_i \cdot) (\pd_{t_i} \phi_k + \lambda \phi_j \cdot \phi_k) \\
& = & \pd_{t_i}\pd_{t_j} \phi_k
+ \lambda (\phi_i \cdot \pd_{t_j} \phi_k + \pd_{t_i} (\phi_j \cdot \phi_k) )
+ \lambda^2 \phi_i \cdot \phi_j \cdot \phi_k \\
& = & \pd_{t_i}\pd_{t_j} \phi_k
+ \lambda (\phi_i \cdot \pd_{t_j} \phi_k + \phi_j \cdot \pd_{t_i} \phi_k + \pd_{t_i} \phi_j \cdot \phi_k ) \\
& + & \lambda^2 \phi_i \cdot \phi_j \cdot \phi_k \\
& = &  \pd_{t_j}\pd_{t_i} \phi_k
+ \lambda (\phi_j \cdot \pd_{t_i} \phi_k + \phi_i \cdot \pd_{t_j} \phi_k + \pd_{t_j} \phi_i \cdot \phi_k ) \\
& + & \lambda^2 \phi_j \cdot \phi_i \cdot \phi_k \\
& = & \nabla^\lambda_{\frac{\pd}{\pd t_j}} \nabla^\lambda_{\frac{\pd}{\pd t_i}}  \phi_k.
\een
\end{proof}

\subsection{The $n$-point functions and the mirror symmetry in genus $0$}

We define the one-point function in genus zero on the big phase space by:
\be
\corrr{\phi_j }_0(\bu) =  \frac{1}{(2j+3)!!} \res (L^{2j+3} dx),
\ee
and define $n$-point function ($n \geq 2$) by taking derivatives:
\be
\corrr{\phi_{j_1}, \cdots, \phi_{j_n}}_0(\bu)= \frac{\pd}{\pd t_{j_1}} \corrr{\phi_{j_2} \cdots \phi_{j_n}}_0(\bu).
\ee

The following result tells us how to recover the genus zero free energy of the 2D topological gravity
from the special deformation of the Airy curve.

\begin{thm}
We have the following identification of $n$-point functions:
\be
\corrr{\phi_{j_1}, \cdots, \phi_{j_n}}_0(\bu)
= \frac{\pd^n F_0}{\pd t_{j_1} \cdots \pd t_{j_n}}(\bt).
\ee
\end{thm}

\begin{proof}
It suffices to prove the one-point case as follws:
\ben
\corrr{\phi_n}_0(\bu)
& = & \frac{1}{(2n+3)!!} \res (L^{2n+3} dx)
= - \frac{1}{(2n+3)!!} \res(x d L^{2n+3}) \\
& = & - \frac{1}{(2n+1)!!} \res \biggl( L^{2n+2}
( L - \sum_{m \geq 0} (2m+1) u_m L^{2m-1} \\
&& - \sum_{m \geq 0} \frac{\pd F_0}{\pd u_m}(\bu) \cdot L^{-2m-3}) d L
\biggr) \\
& = & \frac{1}{(2n+1)!!}  \frac{\pd F_0}{\pd u_n}(\bu)
= \frac{\pd F_0}{\pd t_n}(\bt).
\een
\end{proof}

\section{Quantum Deformation Theory of the Airy Curve}

\label{sec:Quantum-Deformation-Theory}

We have already shown that the free energy in genus zero of 2D topological quantum gravity
can be used to produce special deformation of the Airy curve,
and how to recover the free energy in genus zero of the 2D topological gravity from the deformed
superpotential function.
In this section we will see that this deformation lead to a quantization of the Airy curve that
can be used to recover the free energy in all genera.

\subsection{Symplectic reformulation of the special deformation}

Rewrite \eqref{eqn:X-in-F} as follows:
\be
x(z) = - \sum_{n \geq 0} (2n+1) \tilde{u}_n z^{\frac{2n-1}{2}}
- \sum_{n \geq 0} \frac{\pd F_0}{\pd \tilde{u}_n}(\bu) \cdot z^{-\frac{2n+3}{2} },
\ee
where $z = 2y = f^2$, and
\be
\tilde{u}_n = u_n - \frac{1}{3} \delta_{n,1}.
\ee
One can formally understand $x$ as a field on the Airy curve.

Consider the space of
\be
V = z^{1/2} \bC[z].
\ee
We write an element in $V$ as
\be
\sum_{n =0}^\infty (2n+1) \tilde{u}_n z^{(2n-1)/2}
+ \sum_{n =0}^\infty \tilde{v}_n z^{-(2n+3)/2}
\ee
We regard $\{ \tilde{u}_n, \tilde{v}_n\}$ as linear coordinates on $V$,
and introduce the following symplectic structure on $V$:
\be
\omega = \sum_{n =0}^\infty   d\tilde{u}_n \wedge d \tilde{v}_n.
\ee
It follows that
\be
\tilde{v}_n = \frac{\pd F_0}{\pd u_n}(\bu)
\ee
defines a Lagrangian submanifold in $V$,
and so does
\be
\tilde{v}_n = \frac{\pd (\lambda^2F)}{\pd u_n}(\bu).
\ee
In other words,
free energy in all genera produces a deformation of a Lagrangian sumanifold.

\subsection{Canonical quantization of the special deformation of Airy curve}

Take the natural polarization that $\{q_n = \tilde{u}_n\}$ and $\{p_n = \tilde{v}_n\}$,
one can consider the canonical quantization:
\be
\hat{\tilde{u}}_n = \tilde{u}_n \cdot, \;\;\; \hat{\tilde{v}}_n = \frac{\pd}{\pd \tilde{u}_n}.
\ee
Corresponding to the field $x$,
consider the following fields of operators on the Airy curve:
\be \label{Def:hat(x)}
\hat{x} = - \sum_{n=0}^\infty  \beta_{-(2n+1)} f^{2n-1}
- \sum_{n=0}^\infty  \beta_{2n+1} f^{-2n-3} ,
\ee
where the operators $\beta_{2k+1}$ are defined by:
\be
\beta_{-(2k+1)} = (2k+1) \tilde{u}_k \cdot, \;\;\;\; \beta_{2k+1} = \frac{\pd}{\pd \tilde{u}_k}.
\ee
It is better to write $\hat{x}$ in the $z$-coordinate:
\be
\hat{x}(z) = - \sum_{m \in \bZ} \beta_{-(2m+1)} z^{m-1/2}
= - \sum_{m \in \bZ} \beta_{2m+1} z^{-m-3/2}
\ee

\subsection{The $2$-Reduced bosonic Fock space}

As usual,
the operators $\{ \beta_{2n+1}\}_{n \geq 0}$ are called annihilators
while the operators $\{\beta_{-(2n+1)}\}_{n\geq 0}$ are called creators.
Given $\beta_{2n_1+1}, \dots, \beta_{2n_k+1}$,
their normally ordered products  are defined:
\be
:\beta_{2n_1+1}, \dots, \beta_{2n_k+1}:
= \beta_{2n_1'+1} \cdots \beta_{2n_k'+1},
\ee
where $n_1' \geq \cdots \geq n_k'$ is a rearrangement of $n_1, \dots, n_k$.
Denote $\vac$ the vector $1$, and
by $\Lambda^{(2)}$ the space spanned by elements of form
$\beta_{-(2n_1+1)} \cdots \beta_{-(2n_k+1)} \vac$.
We will refer to $\Lambda^{(2)}$ as the $2$-reduced bosonic Fock space.
On this space one can define a Hermitian product by setting
\bea
&& \langle0 | 0 \rangle = 1, \\
&& \beta_{2n+1}^* = \beta_{-(2n+1)}.
\eea
For a linear operator $A: \Lambda^{(2)} \to \Lambda^{(2)}$,
its vacuum expectation value is defined by:
\be
\corr{A} = \lvac A \vac.
\ee

\subsection{Regularized products of two fields}

We now study the product of the fields $\hat{x}(z)$ with $\hat{x}(z)$.
This cannot be defined directly,
because for example,
\be
\lvac \hat{x}(z) \hat{x}(z) \vac
= \sum_{n \geq 0} (2n+1) = + \infty.
\ee
To fix this problem,
we follow the common practice in the physics literature by using the normally ordered products of fields
and regularization of the singular terms as follows.
First note
\ben
\hat{x}(z) \cdot \hat{x}(w)
& = & :\hat{x}(z)\hat{x}(w):
+ \sum_{n=0}^\infty (2n+1) z^{-(2n-3)/2}w^{(2n-1)/2} \\
& = & :\hat{x}(z)\hat{x}(w): + \frac{z+w}{\sqrt{zw}(z-w)^2}.
\een
It follows that
\be
\corr{ \hat{x}(z) \cdot \hat{x}(w) }
= \frac{z+w}{\sqrt{zw}(z-w)^2},
\ee
hence
\be
\hat{x}(z) \cdot \hat{x}(w) = :\hat{x}(z) \cdot \hat{x}(w): + \corr{\hat{x}(z) \cdot \hat{x}(w)}.
\ee
Now we have
\ben
\hat{x}(z+\epsilon) \cdot \hat{x}(z)
& = & :\hat{x}(z+\epsilon)\hat{x}(z): + \frac{2z+\epsilon}{\sqrt{z(z+\epsilon)}\epsilon^2} \\
& = & :\hat{x}(z+\epsilon)\hat{x}(z): + \frac{2}{\epsilon^2} +  \frac{1}{4z^2} - \frac{\epsilon}{4z^3} + \cdots.
\een
We define the regularized product of $\hat{x}(z)$ with itself by
\be \label{eqn:x(z)odot2}
\begin{split}
\hat{x}(z) \odot \hat{x}(z) & = \hat{x}(z)^{\odot 2}:
= \lim_{\epsilon \to 0} (\hat{x}(z+\epsilon) \hat{x}(z) - \frac{2}{\epsilon^2}) \\
= & :\hat{x}(z)\hat{x}(z): + \frac{1}{4z^2}.
\end{split}
\ee
In other words,
we simply remove the term that goes to infinity as $\epsilon \to 0$,
and then take the limit.

\subsection{Virasoro constraints and mirror symmetry for 2D topological gravity}

By straightforward calculations one can get:

\begin{prop}
Let $L(z) = \sum_{n \in \bZ} L_n z^{-n-2}$ be defined by:
\be
L(z):=\frac{1}{8} :\hat{x}(z)^2: = \frac{1}{8} \sum_{n \in \bZ} \sum_{j+k=-n-1} :\beta_{-(2j+1)} \beta_{-(2k+1)}: z^{-n-2}.
\ee
Then one has the following commutation relations:
\be
[L_m, L_n] = (m-n) L_{m+n} + \frac{2m^3+m}{48} \delta_{m, -n}.
\ee
\end{prop}

Recall the special deformation of the Airy curve constructed in \S \ref{sec:Deformation},
using the genus zero free energy of the 2D topological gravity,
is characterized by (Theorem \ref{thm:Existence} and Theorem \ref{thm:Uniqueness}):
\be
(x^2)_- = 0.
\ee
We take the quantization of this equation to
\be
(\hat{x}(2)^{\odot 2})_- Z = 0.
\ee
The following result establish the mirror symmetry of the theory of 2D topological gravity
and the quantum deformation theory of the Airy curve:

\begin{thm}
The Witten-Kontsevich tau-function $Z_{WK}$ satisfies the following equation:
\be \label{eqn:Virasoro-Operator-Field}
(\hat{x}(2)^{\odot 2})_- Z_{WK} = 0.
\ee
\end{thm}

\begin{proof}
By the definition of $\hat{x}(x)$ and \eqref{eqn:x(z)odot2},
one gets:
\ben
\hat{x}(z)^{\odot 2}
& = & 2 (\half \beta_{-1}^2 + \sum_{n=0}^\infty \beta_{-(2n+3)}\beta_{2n+1} ) z^{-1} \\
& + & 2 \sum_{n =0}^\infty (\beta_{-(2n+1)}\beta_{2n+1}+\frac{1}{8} ) z^{-2} \\
& + & 2 \sum_{m \geq 1} ( \sum_{n=0}^\infty \beta_{-(2n+1)} \beta_{2n+2m+1}
+ \half \sum_{j+k=m-1} \beta_j\beta_k ) z^{-m-2}.
\een
It is then straightforward to see that \eqref{eqn:Virasoro-Operator-Field} is equivalent to
the Virasoro constraints \eqref{eqn:Virasoro-1}-\eqref{eqn:VirasoroN}.
\end{proof}

\section{Regularized Products of Quantum Fields on the Airy Curve}
\label{sec:W-Constraints}

In last section we have defined the regularized product $\hat{x}(z)^{\odot 2}$,
and identify its coefficients with operators of Virasoro constraints for 2D topological gravity
discovered in \cite{DVV-Virasoro}.
In this section we will generalize it to  $\hat{x}(z)^{\odot n}$ for $n > 2$ and  conjecture
the higher W-constraints
\be
(\hat{x}^{\odot 2n})_- Z_{WK} = 0
\ee
hold for all $n \geq 1$.

\subsection{Definition of $\hat{x}(z)^{\odot n}$}

One can inductively define $\hat{x}(z)^{\odot n}$:
\be
\hat{x}(z)^{\odot n}= \hat{x}(z) \odot \hat{x}(z)^{\odot n-1}.
\ee
For example,
\ben
&& \hat{x}(z+\epsilon) \cdot \hat{x}(z)^{\odot 2}
= \hat{x}(z+\epsilon) \cdot (:\hat{x}(z)^2: + \frac{1}{4z^2}) \\
& = & :\hat{x}(z+\epsilon)\hat{x}(z)^2: + 2 \corr{\hat{x}(z+\epsilon)\hat{x}(z)} \cdot \hat{x}(z)
+\frac{\hat{x}(z+\epsilon)}{4z^2} \\
& = &  :\hat{x}(z+\epsilon)\hat{x}(z)^2: + 2 \biggl(\frac{2}{\epsilon^2} +  \frac{1}{4z^2} - \frac{\epsilon}{4z^3} + \cdots
\biggr) \cdot \hat{x}(z)
+\frac{\hat{x}(z+\epsilon)}{4z^2},
\een
the regularized product
$\hat{x}(z) \cdot \hat{x}(z)^{\odot 2}$ is defined by:
\ben
\hat{x}(z) \odot \hat{x}(z)^{\odot 2}
= \lim_{\epsilon \to 0} \biggl( \hat{x}(z+\epsilon) \cdot \hat{x}(z)^{\odot 2} - \frac{4}{\epsilon^2} \hat{x}(z) \biggr)
= :\hat{x}(z)^3: +\frac{3\hat{x}(z)}{4z^2}.
\een

By induction one can easily prove that

\begin{thm}
For $n \geq 1$,
\be
\hat{x}(z)^{\odot  n}
= \sum_{j=0}^{[n/2]} \frac{1}{j!} \binom{n}{2, \dots, 2, n-2j} \frac{1}{(4z^2)^j}
:\hat{x}(z)^{n-2j}:.
\ee
\end{thm}

\begin{cor}
For $n \geq 1$,
\be
: \hat{x}(z)^{n}:
= \sum_{j=0}^{[n/2]} \frac{(-1)^j}{j!} \binom{n}{2, \dots, 2, n-2j} \frac{1}{(4z^2)^j}
\hat{x}(z)^{\odot n-2j}.
\ee
\end{cor}

\begin{proof}
Consider the generating series:
\ben
\sum_{n \geq 0} \hat{x}(z)^{\odot n} \frac{t^n}{n!}
& = & \sum_{n \geq 0}  \frac{t^n}{n!}  \sum_{j=0}^{[n/2]} \frac{1}{j!} \binom{n}{2, \dots, 2, n-2j} \frac{1}{(4z^2)^j}
: \hat{x}(z)^{n-2j} : \\
& = & \sum_{j=0}^\infty \frac{t^{2j}}{j!(8z^2)^j}
\sum_{k = 0}^\infty : \hat{x}(z)^k: \frac{t^k}{k!} \\
& = & \exp (\frac{t^{2}}{8z^2})
\sum_{k = 0}^\infty : \hat{x}(z)^k: \frac{t^k}{k!},
\een
and so
\ben
\sum_{k = 0}^\infty : \hat{x}(z)^k: \frac{t^k}{k!}
& = & \exp (-\frac{t^{2}}{8z^2}) \sum_{n \geq 0} \hat{x}(z)^{\odot n} \frac{t^n}{n!} \\
& = & \sum_{n \geq 0} \frac{t^n}{n!}\sum_{j=0}^{[n/2]} \frac{(-1)^j}{j!} \binom{n}{2, \dots, 2, n-2j} \frac{1}{(4z^2)^j}
\hat{x}(z)^{\odot n-2j}.
\een
\end{proof}

\subsection{Relationship with Bessel polynomials}

Using The On-Line Encyclopedia of Integer Sequences,
one sees that
the coefficients
\ben
&& T(n,j) = \frac{1}{j!}  \binom{n}{2, \dots, 2, n-2j}= \frac{n!}{2^jj!(n-2j)!}= \binom{n}{2j}(2j-1)!!
\een
have the following exponential generating series:
\ben
\sum_{n=0}^\infty \sum_{j=0}^{[n/2]} T(n,j) \frac{z^n}{n!} t^j = \exp ( z + \half tz^2).
\een
These numbers are called {\em Bessel numbers} \cite{Choi-Smith} because
if one sets
\be
y_n(x) = \sum_{k=0}^n T(2n-k,n-k) x^{n-k} = \sum_{k=0}^n \frac{(2n-k)!}{2^{n-k}k!(n-k)!} x^{n-k},
\ee
then $y_n$ is the $n$-th Bessel polynomial that satisfies the Bessel equation:
\be
x^2y''_n +(2x+2)y'_n=n(n+1)y_n.
\ee

\subsection{Regularized products of $\hat{x}^{\odot m}$ with $\hat{x}^{\odot n}$}

Next we define $\hat{x}^{\odot m} \odot \hat{x}^{\odot n}$ for  general $m, n$ in the same fashion.
Let us first define $:\hat{x}(z)^m: \odot :\hat{x}(z)^n:$.
We will first assume that $m \geq n$,
then by Wick's theorem,
\ben
&& :\hat{x}(z+\epsilon)^m: \cdot :\hat{x}(z)^n: = :\hat{x}(z+\epsilon)^m \hat{x}(z)^n: \\
& + & mn :\hat{x}(z+\epsilon)^{m-1} \hat{x}(z)^{n-1}: \cdot \corr{\hat{x}(z+\epsilon) \hat{x}(z)} \\
& + & 2!\cdot \binom{m}{2}\cdot \binom{n}{2} :\hat{x}(z+\epsilon)^{m-2} \hat{x}(z)^{n-2}:
\cdot \corr{\hat{x}(z+\epsilon) \hat{x}(z)}^2 \\
& + & \cdots\cdots \cdots \\
& + & n!\cdot \binom{m}{n}\cdot \binom{n}{n} :\hat{x}(z+\epsilon)^{m-n} \hat{x}(z)^{n-n}:
\cdot \corr{\hat{x}(z+\epsilon) \hat{x}(z)}^n.
\een
We have a similar expression for $m < n$.
Then after removing all the obvious singularities from
$\corr{\hat{x}(z+\epsilon) \hat{x}(z)}^j$ for $j=1, \dots, n$,
and taking $\epsilon \to 0$,
one gets:
\ben
&& :\hat{x}(z)^m: \odot :\hat{x}(z)^n: = \sum_{j=0}^{\min\{m,n\}} \frac{j!}{(4z^2)^j} \cdot
\binom{m}{j}\binom{n}{j} :\hat{x}(z)^{m+n-2j}:.
\een

\begin{thm}
For $m, n \geq 1$,
\be
\hat{x}^{\odot m} \odot \hat{x}^{\odot n} = \hat{x}^{\odot m+ n}.
\ee
\end{thm}

\begin{proof}
We have
\ben
&& \sum_{m \geq 0} :\hat{x}(z)^m: \frac{t_1^m}{m!} \odot
\sum_{n\geq 0} :\hat{x}(z)^n: \frac{t_2^n}{n!} \\
& = & \sum_{m, n \geq 0} \frac{t_1^m}{m!} \frac{t_2^n}{n!}
\sum_{j=0}^{\min\{m,n\}} \frac{j!}{(4z^2)^j}  \cdot
\binom{m}{j}\binom{n}{j} :\hat{x}(z)^{m+n-2j}: \\
& = & \sum_{m, n, j  \geq 0} \frac{t_1^{m+j}}{m!} \frac{t_2^{(n+j)}}{n!} \frac{:\hat{x}(z)^{m+n}:}{j!(4z^2)^j} \\
& = & \sum_{j \geq 0} \frac{(t_1t_2)^j}{j!(4z^2)^j} \sum_{k \geq 0} (t_1+t_2)^k \frac{:\hat{x}(z)^k:}{k!} \\
& = & \exp (\frac{t_1t_2}{4z^2} ) \cdot \exp(-\frac{(t_1+t_2)^2}{8z^2})
\cdot \sum_{n \geq 0} \hat{x}(z)^{\odot n} \frac{(t_1+t_2)^n}{n!} \\
& = & \exp(-\frac{t_1^2+t_2^2}{8z^2})
\cdot \sum_{n \geq 0} \hat{x}(z)^{\odot n} \frac{(t_1+t_2)^n}{n!}.
\een
It follows that
\ben
&& :\hat{x}(z)^m: \odot :\hat{x}(z)^n: = \sum_{j=0}^{\min\{m,n\}} \frac{j!}{(4z^2)^j} \cdot
\binom{m}{j}\binom{n}{j} :\hat{x}(z)^{m+n-2j}:.
\een
It follows that
\ben
&& \sum_{m \geq 0} \hat{x}(z)^{\odot m} \frac{t_1^m}{m!} \odot
\sum_{n\geq 0} \hat{x}(z)^{\odot n} \frac{t_2^n}{n!}
= \sum_{n \geq 0} \hat{x}(z)^{\odot n} \frac{(t_1+t_2)^n}{n!}.
\een
This completes the proof.
\end{proof}

\begin{cor}
The regularized product $\odot$ is associative,
i.e.,
\be
(\hat{x}(z)^{\odot l} \odot \hat{x}(z)^{\odot m}) \odot \hat{x}(z)^{\odot n} \\
= \hat{x}(z)^{\odot l} \odot (\hat{x}(z)^{\odot m} \odot \hat{x}(z)^{\odot n}).
\ee
\end{cor}

\subsection{W-constraints}

We make the following

\begin{conj}
The equalities
\be
(\hat{x}(z)^{\odot 2n})_- Z_{WK} = 0
\ee
hold for all $n \geq 1$.
\end{conj}

{\em Acknowledgements}. 
This research is partially supported by NSFC grant 1171174.


\begin{thebibliography}{999}

\bibitem{ADKMV}
M. Aganagic, R. Dijkgraaf, A. Klemm, M. Mari\~no, C. Vafa,
Topological strings and integrable hierarchies. Comm. Math. Phys.  261  (2006),  no. 2, 451-516.

\bibitem{AKMV}
M. Aganagic, A. Klemm, M. Mari\~no, C. Vafa,
{\em The topological vertex}. Comm. Math. Phys.  254  (2005),  no. 2, 425-478.

\bibitem{AMV}
M. Aganagic, M. Mari\~no, C. Vafa,
{\em All loop topological string amplitudes from Chern-Simons theory}. Comm. Math. Phys.  247  (2004),  no. 2, 467-512.

\bibitem{BCSW}
J. Bennett, D. Cochran, B. Safnuk, K. Woskoff,
{\em Topological recursion for symplectic volumes of moduli spaces of curves}. Mich. Math. J. 61(2) (2012), 331-358.

\bibitem{BCOV1}
M. Bershadsky, S. Cecotti, H. Ooguri, C.  Vafa,
{\em Holomorphic anomalies in topological field theories}. Nuclear Phys. B  405  (1993),  no. 2-3, 279-304.

\bibitem{BCOV2}
M. Bershadsky, S. Cecotti, H. Ooguri, C.  Vafa,
{\em Kodaira-Spencer theory of gravity and exact results for quantum string amplitudes}.
Comm. Math. Phys.  165  (1994),  no. 2, 311-427.

\bibitem{BKMP}
V. Bouchard, A. Klemm, M. Mari\~no, S. Pasquetti,
{\em Remodeling the B-model}. Comm. Math. Phys.  287  (2009),  no. 1, 117-178.

\bibitem{CDGP}
P. Candelas, X. C. de la Ossa, P.S. Green, L. Parkes,
{\em A pair of Calabi-Yau manifolds as an exactly soluble superconformal theory}.
Nuclear Phys. B  359  (1991),  no. 1, 21-74.

\bibitem{CKYZ}
T.-M. Chiang, A. Klemm, S.-T. Yau, E. Zaslow,
{\em Local mirror symmetry: calculations and interpretations}.
Adv. Theor. Math. Phys.  3  (1999),  no. 3, 495-565.

\bibitem{Choi-Smith}
J.Y. Choi, J.D. Smith,
{\em On the unimodality and combinatorics of Bessel numbers}.
The 2000 Com$^2$MaC Conference on Association Schemes, Codes and Designs (Pohang). Discrete Math.  264  (2003),  no. 1-3, 45-53.

\bibitem{Costello-Li}
K.J. Costello, S. Li,
{\em Quantum BCOV theory on Calabi-Yau manifolds and the higher genus B-model},
arXiv:1201.4501, 2012

\bibitem{Dubrovin}
B. Dubrovin,
{\em Geometry of 2D topological field theories}.
Integrable systems and quantum groups (Montecatini Terme, 1993),  120-348,
Lecture Notes in Math., 1620, Springer, Berlin, 1996.

\bibitem{DVV1}
R. Dijkgraaf, H. Verlinde, E. Verlinde,
{\em Notes on topological string theory and 2D quantum gravity}.
String theory and quantum gravity (Trieste, 1990),  91¨C156, World Sci. Publ., River Edge, NJ, 1991.

\bibitem{DVV2}
R. Dijkgraaf, H. Verlinde, E. Verlinde,
{\em Topological strings in $d<1$}. Nuclear Phys. B  352  (1991),  no. 1, 59-86.

\bibitem{DVV-Virasoro}
R. Dijkgraaf, H. Verlinde, E. Verlinde,
{\em Loop equations and Virasoro constraints in nonperturbative two-dimensional quantum gravity}.
Nuclear Phys. B 348 (1991), no. 3, 435-456.


\bibitem{E-K-Y-Y}
T. Eguchi, H. Kanno, Y. Yamada, S.-K. Yang,
{\em Topological strings, flat coordinates and gravitational descendants}.
Phys. Lett. B  305  (1993),  no. 3, 235-241.

\bibitem{E-Y-Y}
T. Eguchi,  Y. Yamada, S.-K. Yang,
{\em Topological field theories and the period integrals}. Modern Phys. Lett. A  8  (1993),  no. 17, 1627-1637.

\bibitem{Eynard}
B. Eynard,
{\em Recursion between Mumford volumes of moduli spaces}.
Ann. Henri Poincar\'e 12(8), 1431-1447 (2011).

\bibitem{EO}
B. Eynard, N. Orantin
{\em Invariants of algebraic curves and topological expansion}.
Commun. Number Theory Phys. 1 (2007), no. 2, 347-452.


\bibitem{EO2}
B. Eynard, N. Orantin
{\em Computation of open Gromov-Witten invariants for toric Calabi-Yau 3-folds by topological recursion, a proof of the BKMP conjecture},
arXiv:1205.1103.



\bibitem{FJRW}
H. Fan, T. Jarvis, Y. Ruan,
{\em The Witten equation, mirror symmetry and quantum singularity theory},
Ann. Math. 178 (2013), 1-106.



\bibitem{Fukuma-Kawai-Nakayama1}
M. Fukuma, H. Kawai, R. Nakayama,
{\em  Continuum Schwinger-Dyson equations and universal structures in two-dimensional quantum gravity}.
Internat. J. Modern Phys. A 6 (1991), no. 8, 1385-1406.

\bibitem{Fukuma-Kawai-Nakayama2}
M. Fukuma, H. Kawai, R. Nakayama,
Infinite-dimensional Grassmannian structure of two-dimensional quantum gravity.
Comm. Math. Phys. 143 (1992), no. 2, 371-403.

\bibitem{Givental}
A. B. Givental,
{e\ mEquivariant Gromov-Witten invariants}.
Internat. Math. Res. Notices  1996,  no. 13, 613-663.

\bibitem{Givental-Quantization}
A. Givental, {\em Gromov-Witten invariants and quantization of quadratic Hamiltonians},
Dedicated to the memory of I. G. Petrovskii on the occasion of his 100th anniversary,
Mosc.Math. J. {\bf 1} (2001), no. 4, 551-568.

\bibitem{GV}
R. Gopakumar,  C.Vafa,
{\em On the gauge theory/geometry correspondence}.
Adv. Theor. Math. Phys.  3  (1999),  no. 5, 1415-1443.



\bibitem{Guo-Zhou1}
S. Guo, J. Zhou,
{\em Gopakumar-Vafa BPS invariants, Hilbert schemes and quasimodular forms. I}
arXiv:1208.3270.

\bibitem{Guo-Zhou2}
S. Guo, J. Zhou,
{\em
Gopakumar-Vafa BPS invariants, Hilbert schemes and quasimodular forms. II}.
Preprint, 2013.

\bibitem{Hori-Vafa}
K. Hori, C. Vafa,
Mirror symmetry, arXiv:hep-th/0002222.

\bibitem{Hua-Kle-Qua}
M.-x. Huang, A. Klemm, S. Quackenbush,
{\em Topological string theory on compact Calabi-Yau: modularity and boundary conditions}.
Homological mirror symmetry,  45-102, Lecture Notes in Phys., 757, Springer, Berlin, 2009.

\bibitem{Kontsevich} M. Kontsevich, {\em
Intersection theory on the moduli space of curves and the matrix
Airy function}. Comm. Math. Phys. {\bf 147} (1992), no. 1, 1--23.

\bibitem{Kontsevich-Manin}
M. Kontsevich, Y. Manin,
{\em Gromov-Witten classes, quantum cohomology, and enumerative geometry}. Comm. Math. Phys.  164  (1994),  no. 3, 525-562.

\bibitem{LLY}
B. Lian, K. Liu, S.-T. Yau,
{\em Mirror principle}. I. Asian J. Math.  1  (1997),  no. 4, 729-763.

\bibitem{LLZ}
 C.-C. M. Liu, K. Liu, J. Zhou
{\em A proof of a conjecture of Mari\~no-Vafa on Hodge integrals}. J. Differential Geom.  65  (2003),  no. 2, 289-340.

\bibitem{LLZ2}
 C.-C. M. Liu, K. Liu, J. Zhou
{\em A formula of two-partition Hodge integrals}. J. Amer. Math. Soc.  20  (2007),  no. 1, 149-184.

\bibitem{LLLZ}
J. Li, C.-C. M. Liu, K. Liu, J. Zhou,
{\em A mathematical theory of the topological vertex}. Geom. Topol.  13  (2009),  no. 1, 527-621.


\bibitem{Losev}
A. Losev,
{\em Descendants constructed from matter field in Landau-Ginzburg theories coupled to topological gravity}.  Teoret. Mat. Fiz.  95  (1993),  no. 2, 307-316;
translation in  Theoret. and Math. Phys.  95  (1993),  no. 2, 595-603

\bibitem{Marino}
M. Mari\~no,
{\em Open string amplitudes and large order behavior in topological string theory}.
J. High Energy Phys.  2008,  no. 3, 060, 34 pp.

\bibitem{Mukherjee}
A. Mukherjee,
{\em 2D gravity coupled to topological minimal models}.
Internat. J. Modern Phys. A  10  (1995),  no. 32, 4601-4632.

\bibitem{Peng}
P. Peng,
{\em A simple proof of Gopakumar-Vafa conjecture for local toric Calabi-Yau manifolds}. Comm. Math. Phys.  276  (2007),  no. 2, 551-569.


\bibitem{Shadrin}
S. V. Shadrin,
{\em  Geometry of meromorphic functions and intersections on moduli spaces of curves}. Int. Math. Res. Not. 2003, no. 38, 2051-2094.

\bibitem{Witten-CS-Jones}
E. Witten,
{\em  Quantum field theory and the Jones polynomial}. Comm. Math. Phys.  121  (1989),  no. 3, 351-399.

\bibitem{Witten1} E. Witten,
{\it Two-dimensional gravity and intersection theory on moduli
space}, Surveys in Differential Geometry, vol.1, (1991) 243--310.

\bibitem{Witten-Q-background}
E. Witten,
{\em Quantum background independence in string theory},
arXiv:hep-th/9306122.

\bibitem{Witten-CS-String}
E. Witten,
{\em Chern-Simons gauge theory as a string theory}.
The Floer Memorial Volume
Progress in Mathematics Volume 133,  1995,   637-678.

\bibitem{Zhou1}
J. Zhou,
{\em Topological recursions of Eynard-Orantin type for intersection numbers on moduli spaces of curves}.
Lett. Math. Phys.  103  (2013),  no. 11, 1191-1206.


\bibitem{Zhou2}
J. Zhou,
{\em Intersection numbers on Deligne-Mumford moduli spaces and quantum Airy curve}.
arXiv:1206.5896.

\bibitem{Zinger}
A. Zinger,
{\em The reduced genus 1 Gromov-Witten invariants of Calabi-Yau hypersurfaces}.
J. Amer. Math. Soc.  22  (2009),  no. 3, 691-737.

\end{thebibliography}
\end{document}